\documentclass[11pt,a4paper]{article}
\usepackage{geometry}
\usepackage{graphicx}
\usepackage{bm,array}
\usepackage{comment}
\usepackage{caption}
\usepackage{subcaption}
\usepackage{tablefootnote}
\usepackage{makecell}
\usepackage{booktabs}
\usepackage{multirow}
\usepackage{float}
\usepackage{lscape}
\usepackage{apacite}
\usepackage{xcolor}
\usepackage[normalem]{ulem}
\usepackage{authblk}
\usepackage{amsmath}
\usepackage{amsfonts}
\usepackage{amssymb}
\usepackage[round, sort,comma,authoryear]{natbib}

\usepackage[pdfpagemode={UseOutlines},bookmarks=true,bookmarksopen=true,
bookmarksopenlevel=0,bookmarksnumbered=true,hypertexnames=false,
colorlinks,linkcolor={blue},citecolor={blue},urlcolor={red},
pdfstartview={FitV},unicode,breaklinks=true]{hyperref}
\hypersetup{urlcolor=blue, colorlinks=true}

\usepackage{setspace}
\usepackage{vmargin}
\setmarginsrb{ 1.0in}  
{ 0.6in}  
{ 1.0in}  
{ 0.8in}  
{  20pt}  
{0.25in}  
{9pt}  
{ 0.3in}  

\newtheorem{proposition}{Proposition}

\newenvironment{proof}[1][Proof]{\noindent\textbf{#1.} }{\ \rule{0.5em}{0.5em}}

\usepackage{multirow}
\usepackage{hhline}
\usepackage{footnote}
\usepackage[ruled,vlined]{algorithm2e}
\usepackage{algorithmic}

\newcommand{\transpose}{{\mbox{\tiny T}}}

\newcommand{\cG}{{\mathcal{G}}}

\newcommand{\cS}{{\mathcal{S}}}

\newcommand{\cX}{{\mathcal{X}}}

\newcommand{\cO}{{\mathcal{O}}}
\newcommand{\cY}{{\mathcal{Y}}}

\newcommand{\bbe}{{\textbf{e}}}

\newcommand{\bx}{\textbf{x}}
\newcommand{\by}{\textbf{y}}

\newcommand{\ba}{\textbf{a}}

\newcommand{\bap}{\pmb{\alpha}} 
\newcommand{\bbt}{\pmb{\beta}}

\newcommand{\bbR}{\mathbb{R}}

\usepackage{mathtools}

\usepackage{mathtools}

\newif\ifnotes\notestrue
\def\htien#1{}

\begin{document}

\newcolumntype{C}{>{\centering\arraybackslash}p{4em}}

\title{\textbf{Outer Approximation and  Supper-modular Cuts for Constrained Assortment Optimization under Mixed-Logit Model}}
\author[]{Hoang Giang Pham}
\author[]{Tien Mai}
\affil[]{\it\small
School of Computing and Information Systems, Singapore Management University}
\date{}
\maketitle

\begin{abstract}
In this paper, we study the assortment optimization problem under the mixed-logit customer choice model. While assortment optimization has been a major topic in revenue management for decades, the mixed-logit model is considered one of the most general and flexible approaches for modeling and predicting customer purchasing behavior. Previous work has primarily relied on mixed-integer linear programming (MILP) or second-order cone (CONIC) reformulations, which allow for exact problem solving using off-the-shelf solvers. However, these approaches often suffer from weak continuous relaxations and are slow when solving large instances. Our research addresses the problem by focusing on components of the objective function that can be proven to be monotonically supper-modular and convex. This allows us to derive valid cuts to outer-approximate the nonlinear objective functions. We then demonstrate that these valid cuts can be incorporated into Cutting Plane or Branch-and-Cut methods to solve the problem exactly. Extensive experiments show that our approaches consistently outperform previous methods in terms of both solution quality and computation time.

\end{abstract}


\noindent
\textbf{Notation:}
Boldface characters represent matrices (or vectors), and $a_i$ denotes the $i$-th element of vector $\ba$. We use $[m]$, for any $m\in \mathbb{N}$, to denote the set $\{1,\ldots,m\}$.

\section{Introduction}

Assortment planning in retail refers to the process of deciding which products to offer during a specific time period. It involves selecting the optimal mix of items, considering factors such as profitability, market share, and customer satisfaction. In the recent operations research literature, there has been increasing attention on assortment optimization problems that involve optimizing the product assortment to maximize revenue, as discussed by \cite{Kök2009} in their survey. To address these assortment optimization problems effectively, it’s crucial to create a model that accurately capture  customers’ purchasing behavior. Such a model should account for how product characteristics influence overall demand and consider customers’ tendencies to substitute between different products.

Our study addresses the assortment optimization problem under the mixed-logit model (MMNL), one of the most popular and flexible models in the discrete choice literature \citep{Trai03, McFa81}. It has been shown that the MMNL, along with the ranking preference and cross-nested logit models, can approximate arbitrarily closely any random utility maximization models \citep{McFadden2000, FosgBier09, aouad2018approximability}, making it one of the most appealing modeling approach to use in the context. The problem has been shown to be NP-hard by \cite{Bront2009} and \cite{Rusmevichientong2014}. In addition, \citep{Desir2022} demonstrates that the unconstrained assortment problem under the MMNL model is NP-hard to approximate within \(\mathcal{O}(1/n^{1-\delta})\), for any \(\delta > 0\) and fixed \(n\), where \(n\) represents the number of customer classes. Moreover, when \(n\) is not fixed, the assortment problem is NP-hard to approximate within any constant factor.

To solve the assortment problem to optimality, prior approaches primarily rely on mixed-integer linear programming (MILP) \citep{Bront2009, MENDEZDIAZ2014246}, or second-order cone (CONIC) reformulations \citep{Sen2018}, making it solvable by off-the-shelf solvers such as CPLEX or GUROBI. Unlike previous studies, our approach does not depend on MILP or CONIC reformulations. Instead, we leverage components of the objective function that are demonstrably super-modular and convex. We then incorporate super-modular and outer-approximation cuts into Cutting Plane (CP) or Branch-and-Cut (B\&C) procedures to efficiently solve the problem. It is worth noting that outer-approximation and/or super-modular cuts are typically used for objective functions that are provably convex and/or super-modular \citep{Duran1986,mai2020multicut,Ljubic2018outer}, which is not the case in the context of assortment optimization. Our work marks the first time such cuts are employed to solve the constrained assortment problem to optimality.



\paragraph{Our contributions:}
Our detailed contributions are presented as follows:
\begin{itemize}
    \item[(i)]To efficiently solve the assortment problem, we convert the maximization problem into a minimization one and explore non-linear components of the objective function that can be proven to be super-modular and convex. We then utilize two types of valid cuts—outer-approximation and super-modular cuts—that can be employed in CP or B\&C methods to outer-approximate the nonlinear objective function. Additionally, to quickly obtain a reasonably good solution useful for a warm start step of the CP or B\&C, we leverage this supermodularity property to demonstrate that a simple polynomial-time Greedy Heuristic can guarantee a \((1-1/e)\frac{r_{\min}}{r_{\max}}\) approximation solution to the assortment optimization problem, even when the number of customer classes is not constant, where \(r_{\min}\) and \(r_{\max}\) are the minimum and maximum product prices, respectively.
    \item[(ii)] We conduct extensive experiments using instances of various sizes, including those obtained from the state-of-the-art method \citep{Sen2018} and those generated with large number of customer classes. Our CP and B\&C are compared with the state-of-the-art approach, CONIC formulation, proposed in \cite{Sen2018}. The comparison results show that our approach significantly outperforms other baselines in terms of both solution quality and runtime. 
\end{itemize}
\paragraph{Paper Outlines:}
The paper is organized as follows. Section \ref{sec:review} presents a literature review. Section \ref{sec:problem} presents the problem description and the MILP and CONIC reformulations. Section \ref{sec:CPBC} describes our CP, B\&C and an approximation scheme for the assortment problem. Section \ref{sec:result} provides numerical experiments, and finally, Section \ref{sec:conclusion} concludes.

\section{Literature review}\label{sec:review}

In the context of assortment optimization, the goal is to select an appropriate subset of items from a larger set to offer to customers. This decision aims to maximize an objective (such as expected revenue) while considering customer preferences and their choice behavior. Specifically, it involves modeling customer substitution between products and understanding how product characteristics impact overall demand. The most widely employed model to capture customer behavior is multinomial logit (MNL) model. This model relies on a probabilistic framework to represent individual customer utilities. Essentially, it estimates the likelihood of a customer choosing a particular product from a set of alternatives based on their preferences and characteristics of the products. The pioneering work of assortment planing problem under MNL model is presented by \cite{ryzin1999} and several follow-up studies are introduced by \cite{Mahajan2001}, \cite{Chong2001}, \cite{Cachon2005}, \cite{rusmevichientong2010dynamic}, \cite{Rusmevichientong2012}, \cite{Topaloglu2013}, \cite{LO2019546}, and \cite{Liu2020}.

Despite its popularity, the MNL model is restricted due to its independence from the Independence from Irrelevant Alternatives property, which implies that the ratio of the probabilities of choosing two products is independent of any other alternative. Furthermore, the total market share of an assortment and the substitution rates within that assortment cannot be independently defined \citep{Kok2007}. These properties are shown to not hold in many practical contexts, and a partial remedy for them has been made based on
an extension of the MNL model, called the nested logit model. Recently, \cite{Davis2014}, \cite{Gallego2014}, \cite{Li2015}, \cite{ALFANDARI2021830} study assortment optimization under variants of the nested logit model. Another extension of the MNL is the MMNL model introduced by \cite{BOYD1980367} and \cite{CARDELL1980423}. This generation of the MNL model can overcome the limitations mentioned above and approximate arbitrarily close to any random utility maximization models, as observed by \cite{McFadden2000}. During the last decade, the assortment planning under the MMNL model has been an interesting topic in the operation research and management science such as \cite{Rusmevichientong2014}, \cite{MENDEZDIAZ2014246}, \cite{Feldman2015}, \cite{Sen2018}, and \cite{Desir2022}.

In this paper, we are interested in the constrained assortment problem under the MMNL model, with linear constraints on the set of products in the assortment. These constraints are often referred to as \textit{capacity constraints}. Although  assortment optimization can be polynomially solvable under the MNL model for both capacitated or uncapacitated scenarios \citep{Talluri2004,rusmevichientong2010dynamic,Sumida2021}, this is not the case for assortment optimization under the MMNL and other models. In \cite{Gallego2014} and \cite{Feldman2015}, several constant factor approximation algorithms for capacitated assortment problem under nested logit model are introduced. \cite{CHEN2019} give a near-optimal algorithm for the same problem containing a single constraint across all products. 
Under the Markov chain choice model \citep{Zhang2005,Blanchet2016} where transitions within the Markov chain represent substitutions, \cite{Desir2020} demonstrate that the capacitated assortment problem remains APX-hard, meaning it is NP-hard to approximate this assortment optimization problem to within any constant factor less than 1.
Under MMNL model, \cite{MENDEZDIAZ2014246} design and test a branch-and-cut algorithm for both capacitated and uncapacitated versions. \cite{Sen2018} formulate the problem as a CONIC quadratic mixed-integer program that can directly solves to optimal large instances of capacitated version. Recently, \cite{Desir2022} present near-optimal algorithms for the capacity constrained assortment optimization problem under a large class of parametric choice models including the MMNL, Markov chain, nested logit, and $d$-level nested logit choice models.

Our work is also related to a line of research on competitive facility location where customer behavior is predicted by discrete choice models \citep{Benati2002, Ljubic2018outer, mai2020multicut, Dam2021submodularity}. In these problems, the objective functions are proven to be submodular and/or convex, so outer-approximation and submodular cuts have been actively applied to solve the facility location problems. As mentioned earlier, in assortment optimization, the objective function under the MMNL does not exhibit such submodularity or concavity properties, and to the best of our knowledge, our work marks the first time such properties are leveraged to optimally solve the assortment problem.

\section{Constrained Assortment Problem under the MMNL Model}\label{sec:problem}
\subsection{Problem Formulation}
Let $[m]$ represent the set of available products and $S$ denote an assortment, which is a subset of the products offered by the retailer. The traditional MNL  model is based on the utility a customer derives from purchasing a product. This utility consists of two components: $U_j = u_j + \epsilon_j$ for any product $j \in [m]$, where $u_j$ is the deterministic component and $\epsilon_j$ is a random component assumed to follow a Gumbel distribution with a mean of zero and variance of $\mu^2\pi^2/6$. Let $p_j$ denote the unit price of product $j$ and $v_0$ represent the no-purchase option. Under the MNL model, the probability that a customer purchases product $j \in [m]$ from an assortment $S$ is given by:
\[
\mathcal{P}_j(S) = \frac{v_j}{v_0 + \sum_{k \in S}v_k}
\]
The MMNL model extends the MNL model by assuming that customers belong to $n$ different classes. Let $\rho_i$ be the probability that the demand originates from customer class $i \in [n]$. Denote $v_{ij}$ as the customer preference for customer class $i \in [n]$ and product $j \in [m]$. For any customer class $i \in [n]$, let $v_{i0}$ be the non-purchase preference and $r_{ij}$ be the revenue from product $j \in [m]$ for customer class $i \in [n]$. The expected revenue for a given assortment $S$ can be expressed as a sum of MNL-based revenues:
\[
\sum_{i \in [n]} \rho_i\left [\frac{\sum_{j\in S} r_{ij} v_{ij}}{v_{i0} + \sum_{j\in S}v_{ij}}\right ]
\]
The assortment optimization problem then can be formulated as:
\begin{equation}\label{prob:Assort}\tag{\sf Assort}
    \max_{S} \left\{ F(S) = \sum_{i \in [n]} \rho_i\left [\frac{\sum_{j\in S} r_{ij} v_{ij}}{v_{i0} + \sum_{j\in S}v_{ij}}\right ] \right\}
\end{equation}
It is convenient to formulate \eqref{prob:Assort} as a binary program. To achieve this, we define the decision variables $\bx$ such that $\bx_j$ equals 1 if product $j$ is included in the assortment, and 0 otherwise. The assortment problem can then be formulated as follows:
\[
\max_{\bx\in \{0,1\}^{m}} \quad \left\{ F(\bx) = \sum_{i\in [n]} \rho_i \left [\frac{\sum_{j\in [m]} r_{ij}v_{ij}x_j}{v_{i0} + \sum_{j\in [m]} v_{ij}x_j} \right ]\right\}
\]
It is known that the above problem is NP-hard for $n\geq 2$, even without any constraint \citep{rusmevichientong2014assortment}. Our research focuses on the capacitated assortment problem, where the above assortment problem is extended by adding resource constraints \citep{Sen2018, Desir2022}. Let $[k]$ be the set of resources that may restrict the assortment, $\beta_{kj}$ be the required amount of resource used by product $j \in [m]$ and $\alpha_{k}$ denote the amount of resource $k \in [k]$ available. The capacitated assortment problem is presented as follows:
\begin{align}\label{prob:CAP}\tag{\sf CAP}
    \max &\quad \left\{F(\bx) = \sum_{i\in [n]} \rho_i \left [\frac{\sum_{j\in [m]} r_{ij} v_{ij} x_j}{v_{i0} + \sum_{j\in [m]}v_{ij}x_j} \right ]\right\}\\
    \text{s.t} &\quad \sum_{j \in [m]} \beta_{kj} x_j \leq \alpha_k \qquad \forall k \in [k] \nonumber \\
    &\quad x_j \in \{0,1\} \qquad \forall j \in [m]\nonumber
\end{align}

\subsection{The Traditional MILP and CONIC Reformulations}\label{sec:formulation}
In this section, we  revisit the classic MILP   and CONIC  reformulations for \eqref{prob:CAP}. To this end, let us first reformulate \eqref{prob:CAP} as a minimization problem. 
Define \( r_i = \max_{j \in [m]} r_{ij} \) for all \( i \in [n] \), and \( r'_{ij} = r_i - r_{ij} \) for all \( i \in [n] \) and \( j \in [m] \). For notational simplicity, let \(\cX\) represent the feasible set of \eqref{prob:CAP}. The objective function in \eqref{prob:CAP} can be written as follows:
\[
F(\bx) = \sum_{i \in [n]} \rho_i r_i -  \sum_{i\in [n]} \rho_i \left [ \frac{r_i v_{i0} + \sum_{j\in [m]} r'_{ij} v_{ij} x_j}{v_{i0} + \sum_{j\in [m]} v_{ij} x_j} \right ]
\]
Because the first component of $F(\bx)$ is a constant, we then can convert \eqref{prob:CAP}  as the following minimization problem:
\begin{equation}\label{prob:CAP-min}\tag{\sf CAP-MIN}
    \min_{\bx\in \cX} \left\{ G(\bx) =  \sum_{i\in [n]}\rho_i \left [\frac{\sum_{j\in [m]} r'_{ij} v_{ij} x_j}{v_{i0} + \sum_{j\in [m]}v_{ij}x_j} \right ] + \sum_{i\in [n]} \frac{\rho_i r_i v_{i0}}{v_{i0} + \sum_{j\in [m]}v_{ij}x_j}\right\}
\end{equation}

\paragraph{MILP  Reformulation}
By letting  $y_i = 1 / (v_{i0} + \sum_{j \in [m]} v_{ij}x_j)$ for any $i \in [n]$, we can formulate \eqref{prob:CAP-min} as the following  mixed-integer bilinear program:
\begin{align}\label{prob:BiBC}\tag{\sf Bilinear}
    \min &\quad \sum_{i\in [n]}\sum_{j\in [m]}\rho_i r'_{ij} v_{ij}x_j y_i + \sum_{i\in [n]} \rho_i r_i v_{i0}y_i \notag\\
    \text{s.t.} &\quad \sum_{j \in [m]} \beta_{kj} x_j \leq \alpha_k \qquad \forall k \in [k] \label{eq:capacity}\\
    &\quad v_{i0}y_i + \sum_{j \in [m]}v_{ij}x_jy_i \geq 1 \qquad \forall i \in [n] \label{eq:relation_x_y}\\
    &\quad x_j \in \{0,1\} \qquad \forall j \in [m] \label{eq:x}\\
    &\quad y_i \geq 0 \qquad \forall i \in [n] \label{eq:y_geq_0}
\end{align}
The bilinear term $x_jy_i$ in the formulation can be linearized using additional variables $\theta_{ij} = x_jy_i$ and big-M \citep{Bront2009} constraints:
\begin{align}
    & v_{i0}\theta_{ij} \leq x_j \qquad \forall i \in [n], j \in [m] \label{eq:linear2}\\
    & \theta_{ij} \leq y_i \qquad \forall i \in [n], j \in [m]\label{eq:linear3}\\
    & v_{i0}(y_i - \theta_{ij}) \leq 1 - x_j \qquad \forall i \in [n], j \in [m] \label{eq:linear1}\\
    & \theta_{ij} \geq 0 \qquad \forall i \in [n], j \in [m] \label{eq:theta_geq_0}
\end{align}
It is widely known that big-M techniques for linearizing bilinear terms often result in weak continuous relaxations. In \cite{Sen2018}, the authors propose strengthening such linearizations with McCormick inequalities by using conditional bounds on \( y_i \) for all \( i \in [n] \). Specifically, let us first define the conditional function  \(\phi_{ij}(b)\), for any $i\in [n], j\in [m]$  and $b\in \{0,1\}$ as: 
\[
\phi_{ij}(b) =  \min_{\bx \in \cX,~ x_j = b} \left\{\frac{1}{v_{i0}+\sum_{j\in [m]} v_{ij}x_j}\right\}
\]
The value of \(\phi_{ij}(b)\) can be obtained by solving a linear problem: \(\max_{\bx \in \cX, x_j = b} \sum_{j \in [m]} v_{ij} x_j\), which can be done efficiently under cardinality constraints. In fact, \(\phi_{ij}(b)\) provides a tight lower bound for \(y_i\) conditional on \(x_j = b\). Using these bounds, the bilinear terms \(\theta_{ij} = y_i x_j\) can be linearized using the following McCormick inequalities:
\begin{align}
    & \theta_{ij} \leq \frac{1}{v_{i0} + v_{ij}} x_j \qquad \forall i \in [n], j \in [m] \label{eq:sl4}\\
    & \theta_{ij} \leq y_i - \phi_{ij}(0) (1 - x_j) \qquad \forall i \in [n], j \in [m] \label{eq:sl0}\\
    & \theta_{ij} \geq \phi_{ij}(1) x_j \qquad \forall i \in [n], j \in [m]\label{eq:sl1}\\
    & \theta_{ij} \geq y_i - \frac{1}{v_{i0}} (1 - x_j) \qquad \forall i \in [n], j \in [m] \label{eq:sl2}
\end{align}
Note that constraints (\ref{eq:sl2}) are equivalent to constraints (\ref{eq:linear1}) and constraints (\ref{eq:sl4}), \eqref{eq:sl0} are stronger than constraints (\ref{eq:linear2}), \eqref{eq:linear3}, respectively, thanks  to the tighter bounds $\phi_{ij}$, for all $i\in [n], j\in [m]$. The linearization and the McCormick inequalities lead to the following MILP formulation:
\begin{align}\label{prob:MILPBC}\tag{\sf MILP}
    \min &\quad \sum_{i\in [n]}\sum_{j\in [m]}\rho_i r'_{ij} v_{ij}\theta_{ij} + \sum_{i\in [n]} \rho_i r_i v_{i0}y_i\\
    \text{s.t.} & \quad \sum_{j \in [m]} \beta_{kj} x_j \leq \alpha_k \qquad \forall k \in [k] \nonumber\\
    & \quad \theta_{ij} \leq \frac{1}{v_{i0} + v_{ij}} x_j \qquad \forall i \in [n], j \in [m] \nonumber\\
    & \quad \theta_{ij} \geq \phi_{ij}(1) x_j \qquad \forall i \in [n], j \in [m]\nonumber\\
    & \quad \theta_{ij} \leq y_i - \phi_{ij}(0) (1 - x_j) \qquad \forall i \in [n], j \in [m] \nonumber\\
    & \quad \theta_{ij} \geq y_i - \frac{1}{v_{i0}} (1 - x_j) \qquad \forall i \in [n], j \in [m] \nonumber \\
    & \quad v_{i0}y_i + \sum_{j \in [m]}v_{ij}\theta_{ij} \geq 1 \qquad \forall i \in [n] \label{eq:define_y_milp_rexlax}\\
    & \quad x_j \in \{0,1\} \qquad \forall j \in [m]\nonumber\\
    & \quad y_i, \theta_{ij} \geq 0 \qquad \forall i \in [n], j \in [m]\nonumber
\end{align}

\paragraph{CONIC Reformulation}


The MILP reformulation mentioned above, even with McCormick inequalities, is reported to have poor performance due to weak continuous relaxation. To address this issue, the CONIC reformulation proposed in \citep{Sen2018} leverages rotated cone constraints of the form \(x_1^2 \leq x_2 x_3\) with \(x_1, x_2, x_3 \geq 0\), and reformulates \eqref{prob:CAP-min} as a mixed-integer second-order cone program, which can be readily solved by an off-the-shelf solver. Specifically, by letting \(z_i = v_{i0} + \sum_{j \in [m]} v_{ij} x_j\), the rotated cone form of \(y_i = 1/(v_{i0} + \sum_{j \in [m]} v_{ij} x_j)\) and \(\theta_{ij} = x_j y_i\) can be presented as follows:
\begin{align}
    &y_iz_i \geq 1 \qquad \forall i \in [n] \label{eq:cone_yz}\\
    & \theta_{ij}z_i \geq x_j^2 \qquad \forall i \in [n], j \in [m] \label{eq:cone_thetaz}
\end{align}
The CONIC formulation introduced by \cite{Sen2018} can be written as follows:
\begin{align}\label{prob:Conic}\tag{\sf CONIC}
    \min &\quad \sum_{i\in [n]}\sum_{j\in [m]}\rho_i r'_{ij} v_{ij}\theta_{ij} + \sum_{i\in [n]} \rho_i r_i v_{i0}y_i\\
    \text{s.t.} &\quad \sum_{j \in [m]} \beta_{kj} x_j \leq \alpha_k \qquad \forall k \in [k] \nonumber\\
    &\quad z_i = v_{i0} + \sum_{j \in [m]}v_{ij}x_j \qquad \forall i \in [n] \nonumber\\
    & \quad y_iz_i \geq 1 \qquad \forall i \in [n] \nonumber\\
    & \quad \theta_{ij}z_i \geq x_j^2 \qquad \forall i \in [n], j \in [m] \nonumber\\
    & \quad v_{i0}y_i + \sum_{j \in [m]}v_{ij}\theta_{ij} \geq 1 \qquad \forall i \in [n] \label{ctr:mc-0}\\
    & \quad \theta_{ij} \leq \frac{1}{v_{i0} + v_{ij}} x_j \qquad \forall i \in [n], j \in [m] \label{ctr:mc-1}\\
    & \quad \theta_{ij} \geq \phi_{ij}(1) x_j \qquad \forall i \in [n], j \in [m]\label{ctr:mc-2}\\
    & \quad \theta_{ij} \leq y_i - \phi_{ij}(0) (1 - x_j) \qquad \forall i \in [n], j \in [m] \label{ctr:mc-3}\\
    & \quad \theta_{ij} \geq y_i - \frac{1}{v_{i0}} (1 - x_j) \qquad \forall i \in [n], j \in [m] \label{ctr:mc-4}\\
    & \quad x_j \in \{0,1\} \qquad \forall j \in [m]\nonumber\\
    & \quad y_i, z_i, \theta_{ij} \geq 0 \qquad \forall i \in [n], j \in [m]\nonumber
\end{align}
 It is important to note that Constraints \eqref{ctr:mc-0}-\eqref{ctr:mc-4} are McCormick inequalities and are redundant in the formulation. However, \cite{Sen2018} suggest including these constraints to enhance the continuous relaxations of the mixed-integer nonlinear  program.

\section{Outer Approximation and Supper-modular Cuts}\label{sec:CPBC}
In this section, we examine components of the objective function in \eqref{prob:CAP-min} that can be proven to be convex and supper-modular. This analysis will enable us to derive valid cuts that can be incorporated into a CP or B\&C procedure to solve the problem.

\subsection{Supper-modularity and Convexity}
For each $i\in [n]$ and $j\in [m]$, let us denote 
\begin{align}
    \Phi_i(\bx)  &= 1 / (v_{i0} + \sum_{j \in [m]} v_{ij}x_j) \nonumber\\
    \Psi_i(\bx) &= v_{i0} + \sum_{j \in [m]}v_{ij}x_j
\end{align}
We also define $\Phi_i()$ as a  subset function if it takes input as an assortment solution:
\begin{align}
    \Phi_i(S)  &= 1 / (v_{i0} + \sum_{j \in S} v_{ij}) \nonumber
\end{align}
We start by stating the following  proposition:
\begin{proposition}
For any $i \in [n]$, $\Phi_i(\bx)$ is convex in $\bx$ and $\Phi_i(S)$ is monotonically decreasing and supper-modular.
\end{proposition}
The proposition can be easily verified by noting that \((v_{i0} + \sum_{j \in [m]} v_{ij} x_j)\) is linear in \(\bx\), and the function \(f(t) = 1/t\) is convex in \(t\). As a result, \(\Phi_i(\bx)\) is convex in \(\bx\). The submodularity can also be verified by observing that, for any subsets \(A \subset B \subset [m]\) and any \(j^* \in [m] \setminus B\), we have the following inequalities:
\begin{align}
    \Phi_i(A) - \Phi_i(A \cup \{j^*\}) &= \frac{1}{v_{i0} + \sum_{j \in A} v_{ij}} - \frac{1}{v_{i0} + \sum_{j \in A \cup \{j^*\}} v_{ij}} \nonumber\\
    &= \frac{v_{ij^*}}{(v_{i0} + \sum_{j \in A} v_{ij})(v_{i0} + \sum_{j \in A \cup \{j^*\}} v_{ij})} \nonumber\\
    &\geq \frac{v_{ij^*}}{(v_{i0} + \sum_{j \in B} v_{ij})(v_{i0} + \sum_{j \in B \cup \{j^*\}} v_{ij})} \nonumber\\
    &= \Phi_i(B) - \Phi_i(B \cup \{j^*\}),
\end{align}
which directly implies the supper-modularity.

The convexity then implies that, for any solution candidate $\Bar{\bx} \in \{0,1\}^m$, the following  inequalities are valid for any $\bx \in \cX$:
\begin{align}
     \Phi_i(\bx) \geq \sum_{j \in [m]}\frac{-v_{ij}}{\Psi_i(\Bar{\bx})^2}(x_j - \Bar{x}_j) + \Phi_i(\Bar{\bx}) \qquad \forall i \in [n] \label{eq:oa_cut}
\end{align}
Moreover, since $\Phi_i(S)$ is supper-modular, the following inequality is valid for any $\Bar{S}\subset [m]$:
\begin{align}
    \Phi_i(S) &\geq \sum_{k\in [m]\backslash \Bar{S}\cap S}\psi^i_{k}(\Bar{S}) - \sum_{k\in \Bar{S}\backslash S} \psi_k^i([m]-k) + \Phi_i(\Bar{S}) \notag\\
     \Phi_i(S) &\geq \sum_{k\in [m]\backslash \Bar{S}\cap S}\psi^i_{k}(\varnothing) - \sum_{k\in \Bar{S}\backslash S} \psi_k^i(\Bar{S}-k) + \Phi_i(\Bar{S}) \notag
\end{align}
where $\psi^i_k(S) = \phi_i(S+k) - \phi_i(S)$. These can be transformed into inequalities in the $\bx$ space as:
\begin{align}
      \Phi_i(\bx) &\geq \sum_{k\in [m]}\frac{v_{ik}(\Bar{x}_k - 1)}{\Psi_i(\Bar{\bx})\Psi_i(\Bar{\bx}+\bbe_k)}x_k + \sum_{k\in [m]} \frac{v_{ik}\Bar{x}_k}{\Psi_i(\bbe)\Psi_i(\bbe-\bbe_k)}(1-x_k) + \Phi_i(\Bar{\bx}) \label{eq:supper-modularcut1}\\
      \Phi_i(\bx) &\geq \sum_{k\in [m]}\frac{v_{ik}(\Bar{x}_k - 1)}{\Psi_i(\textbf{0})\Psi_i(\textbf{0}+\bbe_k)}x_k + \sum_{k\in [m]} \frac{v_{ik}\Bar{\bx}_k}{\Psi_i(\Bar{\bx})\Psi_i(\Bar{\bx}-\bbe_k)}(1-x_k) + \Phi_i(\Bar{\bx}) \label{eq:supper-modularcut2}
\end{align}
for any $\bx\in \cX$, where $\bbe$ is an all-one vector of size $m$, $\bbe_k$ is a vector of size $m$ with zero elements except the $k$-th element which takes a value of 1, and $\textbf{0}$ is an all-zero vector of size $m$.

\subsection{Cutting Plane and Branch-and-Cut Approaches}
Both CP and B\&C procedures work by defining a master problem, which can be solved quickly using an off-the-shelf solver. At each step, valid cuts are added to the master problem to outer-approximate the nonlinear terms, and the master problem is repeatedly solved until an optimal solution is found or a stopping condition is met. In our context, the entire objective function is neither convex nor submodular, so typical CP or B\&C approaches are not directly applicable. To utilize the valid cuts discussed above, let us rewrite \eqref{prob:CAP-min} as the following mixed-integer nonlinear program:
\begin{align}
    \min_{\bx\in \cX} &\quad \left\{\sum_{i\in [n]}\sum_{j\in [m]}\rho_i r'_{ij} v_{ij}x_j \Phi_i(\bx) + \sum_{i\in [n]} \rho_i r_i v_{i0}\Phi_i(\bx) \right\}\label{prob-1}
\end{align}
Since all the terms \(\rho_i r'_{ij} v_{ij} x_j\) and \(\rho_i r_i v_{i0}\) are non-negative for all \(i \in [n]\) and \(j \in [m]\), we can set \(y_i = \Phi_i(\bx)\) and reformulate \eqref{prob-1} as:
\begin{align}\label{prob-2}\tag{\sf MINLP}
    \min &\quad \sum_{i \in [n]} \sum_{j \in [m]} \rho_i r'_{ij} v_{ij} x_j y_i + \sum_{i \in [n]} \rho_i r_i v_{i0} y_i \\
    \text{s.t.} & \quad y_i \geq \Phi_i(\bx),~\forall i \in [n] \label{ctr:Phi-y} \\
    &\quad \bx \in \cX, \by \in \bbR^n_+
\end{align}
To support the use of the outer-approximation and submodular cuts mentioned above, the following proposition states that the mixed-integer nonlinear program in \eqref{prob-2} is equivalent to a mixed-integer bilinear program where the nonlinear constraints \eqref{ctr:Phi-y} are replaced by a set of valid cuts (outer-approximation or submodular cuts).

\begin{proposition}\label{prop:MIBLP-equiv}
    The MINLP is equivalent to the following mixed-integer bilinear program:
    \begin{align}\label{prob-3}\tag{\sf MIBLP}
        \min &\quad \sum_{i \in [n]} \sum_{j \in [m]} \rho_i r'_{ij} v_{ij} x_j y_i + \sum_{i \in [n]} \rho_i r_i v_{i0} y_i \\
        \text{s.t.} & \quad y_i \geq (\bap^i_{\overline{\bx}})^\transpose \bx + \bbt^i_{\overline{\bx}},~\forall \overline{\bx} \in \cX,~i \in [n] \label{ctr:Phi-y-cut} \\
        &\quad \bx \in \cX, \by \in \bbR^n_+
    \end{align}
where \(y_i \geq (\bap^i_{\overline{\bx}})^\transpose \bx + \bbt^i_{\overline{\bx}}\) are outer-approximation or submodular cuts, constructed by generating valid inequalities of the forms in \eqref{eq:oa_cut}, \eqref{eq:supper-modularcut1}, and \eqref{eq:supper-modularcut2} at point \(\overline{\bx}\).
\end{proposition}
\begin{proof}
 We will prove that the feasible sets of \eqref{prob-2} and \eqref{prob-3} are equivalent. To this end, let \((\bx, \by)\) be a feasible solution to \eqref{prob-2}, implying \(y_i \leq \Phi_i(\bx)\) for all \(\bx \in \cX\). Since all the cuts in \eqref{prob-3} are valid, we have \(\Phi_i(\bx) \leq (\bap^i_{\overline{\bx}})^\transpose \bx + \bbt^i_{\overline{\bx}}\) for any \(\overline{\bx} \in \cX\). This leads to 
\[
y_i \leq (\bap^i_{\overline{\bx}})^\transpose \bx + \bbt^i_{\overline{\bx}},~\forall \overline{\bx} \in \cX,
\]
implying that \((\bx, \by)\) is also feasible for \eqref{prob-3}. 

For the opposite direction, let \((\bx, \by)\) be a feasible solution to \eqref{prob-3}. We have, for any \(\bx \in \cX\),
\[
y_i \leq \min_{\overline{\bx} \in \cX} (\bap^i_{\overline{\bx}})^\transpose \bx + \bbt^i_{\overline{\bx}} \leq (\bap^i_{\bx})^\transpose \bx + \bbt^i_{\bx} = \Psi_i(\bx),
\]
implying that \((\bx, \by)\) is also feasible for \eqref{prob-2}.
 \end{proof}

It is important to note that prior studies have demonstrated that submodular maximization can be equivalently reformulated as a MILP with constraints given by submodular cuts generated at every point in the feasible set. Similarly, a mixed-integer convex optimization can be reformulated as a MILP with constraints given by outer-approximation cuts generated at every point in the feasible set. The result in Proposition \ref{prop:MIBLP-equiv} shares a similar synergy but is distinct in that the objective in \eqref{prob-2} is neither submodular nor convex.

\paragraph{Cutting Plane.} Proposition \ref{prop:MIBLP-equiv} suggests that we can define a master problem by removing the nonlinear constraints in \eqref{ctr:Phi-y}. At each iteration of the CP, valid cuts (outer-approximation or submodular cuts) are added to the master problem, which is then solved repeatedly to obtain new candidate solutions. The process terminates when a solution that is feasible for \eqref{prob-2} is found. 
It is important to note that this process always terminates after a finite number of iterations and returns an optimal solution to \eqref{prob-2} and \eqref{prob-3}. To understand why, let \(\cY = \{(\bx^1,\by^1), \ldots, (\bx^k,\by^k)\}\) represent the \(k\) solution candidates obtained after \(k\) iterations of the CP. In the next iteration, assume the master problem returns a solution candidate \((\bx^*, \by^*)\). If \((\bx^*, \by^*)\) already appears in \(\cY\), it must be an optimal solution to \eqref{prob-2}. To prove this, assume \((\bx^*, \by^*) \in \cY\). Since the master problem at iteration \(k\) contains valid cuts generated by points in \(\cY\), we have:
\[
{y}^*_i \geq (\bap^i_{\overline{\bx}})^\transpose \bx + \bbt^i_{\overline{\bx}},~\forall \overline{\bx} \in \cY, \bx \in \cX, i \in [n].
\]
Since \((\bx^*, \by^*) \in \cY\), we have:
\[
{y}^*_i \geq (\bap^i_{\bx^*})^\transpose \bx + \bbt^i_{\bx^*} = \Psi_i(\bx^*),~\forall i \in [n],
\]
implying that \((\bx^*, \by^*)\) is feasible for \eqref{prob-2} (which also imply that  the CP process will terminate).  Moreover, the master problem shares the same objective function as \eqref{prob-2}, and its feasible set always contains the feasible set of \eqref{prob-2}. Thus, an optimal solution to the master problem being feasible for \eqref{prob-2} directly implies that \((\bx^*, \by^*)\) is optimal for \eqref{prob-2}. Therefore, we obtain the desired claim. Since the feasible set \(\cX\) is finite, this also indicates that the CP always terminates at an optimal solution after a finite number of iterations.

Our CP will rely on two versions of the master problem: one is the bilinear program in \eqref{prob-3}, and the other is a MILP where the bilinear terms are linearized using McCormick inequalities. Specifically, we will utilize the following master problems:
    \begin{align}\label{Bi-master}\tag{\sf Bi-Master}
        \min &\quad \sum_{i \in [n]} \sum_{j \in [m]} \rho_i r'_{ij} v_{ij} x_j y_i + \sum_{i \in [n]} \rho_i r_i v_{i0} y_i \\
        \text{s.t.} & \quad \texttt{[Outer-approximation \& supper-modular cuts]}\nonumber \\
        &\quad \bx \in \cX, \by \in \bbR^n_+\nonumber
    \end{align}
and a linearized version of \eqref{Bi-master}:
    \begin{align}\label{Li-master}\tag{\sf Li-Master}
        \min &\quad \sum_{i \in [n]} \sum_{j \in [m]} \rho_i r'_{ij} v_{ij} \theta_{ij} + \sum_{i \in [n]} \rho_i r_i v_{i0} y_i \\
        \text{s.t.} & \quad \texttt{[Outer-approximation \& supper-modular cuts]}\nonumber \\
            & \quad \theta_{ij} \leq \frac{1}{v_{i0} + v_{ij}} x_j \qquad \forall i \in [n], j \in [m] \label{ctr:master-mc-1}\\
    & \quad \theta_{ij} \geq \phi_{ij}(1) x_j \qquad \forall i \in [n], j \in [m]\label{ctr:master-mc-2}\\
    & \quad \theta_{ij} \leq y_i - \phi_{ij}(0) (1 - x_j) \qquad \forall i \in [n], j \in [m] \label{ctr:master-mc-3}\\
    & \quad \theta_{ij} \geq y_i - \frac{1}{v_{i0}} (1 - x_j) \qquad \forall i \in [n], j \in [m] \label{ctr:master-mc-4}\\
        &\quad \bx \in \cX, \by \in \bbR^n_+
    \end{align}
where Constraints \eqref{ctr:master-mc-1}-\eqref{ctr:master-mc-4} are McCormick inequalities used to linearize the bilinear term $\theta_{ij}  = x_iy_j$.  

In Algorithm \ref{algo:CP} below, we present the main steps of our CP method utilizing the bilinear master problem in \eqref{Bi-master}. A similar approach can be implemented using the linear master problem in \eqref{Li-master}.

\begin{algorithm}[htb]
    \caption{Cutting Plane} 
    \label{algo:CP}
    \SetKwRepeat{Do}{do}{until}
Set a small $\epsilon>0$ as the optimal gap\\
Use a greedy heuristic  algorithm to get a solution candidate  $\overline{\bx}$\\
Build the master problem and add some initial cuts using $\Bar{\bx}$

\Do{$y_i\geq \Phi_i(\Bar{\bx})$ for all $i\in [n]$
}
{
  Solve the  master problem to get a new solution candidate $(\Bar{\bx},\Bar{\by})$\\
  Add outer-approximation and supper-modular cuts  based on  $(\Bar{\bx},\Bar{\by})$ to the master problem
 }
    Return $(\overline{\bx})$ as an optimal solution.
\end{algorithm}

\paragraph{Segment-based CP.} In the above CP methods, cuts are added for each function \(\Phi_i(\bx)\) corresponding to each customer class \(i \in [n]\). Consequently, the number of cuts added to the master problem at each iteration is proportional to the number of customer classes. This can lead to a rapid increase in the size of the master problem if the number of customer classes is large, which is often the case when, for instance, the choice probabilities of the MMNL model are approximated by sample average approximation. The segment-based outer-approximation cut technique proposed by \cite{MaiLodi2020_OA} suggests that we can combine individual customer classes into larger groups to balance the number of cuts created at each iteration and the number of iterations until convergence. To implement this idea, let us divide the set of \(n\) customer classes into \(L\) disjoint groups \(\cG_1, \ldots, \cG_L\) such that \(\bigcup_{l \in [L]} \cG_l = [n]\). Using this,  \eqref{prob-1} can be rewritten as:
\begin{align}
    \min_{\bx\in \cX} &\quad \left\{\sum_{j\in [m]}\left(\sum_{l\in [L]} x_j G^j_l(\bx)\right) + \sum_{l\in [L]} R_l(\bx) \right\}\label{prob-1-multicut}
\end{align}
where
\begin{align*}
    G^j_l(\bx) &= \sum_{i\in \cG_l} \rho_i r'_{ij} v_{ij}\Phi_i(\bx) \\
    R_l(\bx) &= \sum_{i \in  \cG_l}\rho_i r_i v_{i0}\Phi_i(\bx)  
\end{align*}
We then can reformulate \eqref{prob-1-multicut} as:
\begin{align}\label{prob-2-SB}\tag{\sf MINLP-SB}
    \min &\quad \sum_{i \in [n]} \sum_{l\in [L]} x_j z^j_l + \sum_{l\in [L] } t_i \\
    \text{s.t.} & \quad z^j_l \geq G^j_l(\bx),~\forall l \in [L], j\in [m] \label{ctr:Phi-z} \\
     & \quad t_l \geq R_l(\bx),~\forall l \in [L] \label{ctr:Phi-t} \\
    &\quad \bx \in \cX, \by \in \bbR^n_+
\end{align}
It can be observed that \(G^j_l(\bx)\) and \(R_l(\bx)\) exhibit similar properties to those of \(\Psi_i(\bx)\), namely convexity and supper-modularity. Therefore, outer-approximation and supper-modular cuts of the form in \eqref{eq:oa_cut}, \eqref{eq:supper-modularcut1}, and \eqref{eq:supper-modularcut2} (but derived for \(G^j_l(\cdot)\) and \(R_l(\bx)\)) can be used to approximate the nonlinear functions. Specifically, for any candidate solution $\Bar{\bx}$, the following cuts are valid for \eqref{prob-2-SB}:
\begin{align}
    z^j_l &\geq \sum_{i \in \cG_t} \rho_i r'_i v_{ij}\left [ \sum_{j' \in [m]}\frac{-v_{ij'}}{(v_{i0}+\sum_{j\in [m]}v_{ij}x_j)^2}(x_j - \Bar{x}_j) + \Phi_i(\Bar{\bx}) \right ] \label{eq:SB-OA-z}\\
    t_l &\geq \sum_{i \in \cG_t}\rho_i r_i v_{i0} \left [ \sum_{j' \in [m]}\frac{-v_{ij'}}{(v_{i0}+\sum_{j\in [m]}v_{ij}x_j)^2}(x_j - \Bar{x}_j) + \Phi_i(\Bar{\bx}) \right ] \label{eq:SB-OA-t}\\
      z^j_l &\geq \sum_{i \in \cG_l} \rho_i r'_{ij} v_{ij} \left [\sum_{k\in [m]}\frac{v_{ik}(\Bar{x}_k - 1)}{\Psi_i(\Bar{\bx})\Psi_i(\Bar{\bx}+\bbe_k)}x_k + \sum_{k\in [m]} \frac{v_{ik}\Bar{x}_k}{\Psi_i(\bbe)\Psi_i(\bbe-\bbe_k)}(1-x_k) + \Phi_i(\Bar{\bx})\right ]\label{eq:SB-SC-z-1}\\
      z^j_l &\geq \sum_{i \in \cG_l} \rho_i r'_{ij} v_{ij}\left [ \sum_{k\in [m]}\frac{v_{ik}(\Bar{x}_k - 1)}{\Psi_i(\textbf{0})\Psi_i(\textbf{0}+\bbe_k)}x_k + \sum_{k\in [m]} \frac{v_{ik}\Bar{\bx}_k}{\Psi_i(\Bar{\bx})\Psi_i(\Bar{\bx}-\bbe_k)}(1-x_k) + \Phi_i(\Bar{\bx}) \right ]\label{eq:SB-SC-z-2}\\
           t_l &\geq \sum_{i \in \cG_l} \rho_i r_i v_{i0}  \left [\sum_{k\in [m]}\frac{v_{ik}(\Bar{x}_k - 1)}{\Psi_i(\Bar{\bx})\Psi_i(\Bar{\bx}+\bbe_k)}x_k + \sum_{k\in [m]} \frac{v_{ik}\Bar{x}_k}{\Psi_i(\bbe)\Psi_i(\bbe-\bbe_k)}(1-x_k) + \Phi_i(\Bar{\bx})\right ]\label{eq:SB-SC-t-1}\\
      t_l &\geq \sum_{i \in \cG_l} \rho_i r_i v_{i0} \left [ \sum_{k\in [m]}\frac{v_{ik}(\Bar{x}_k - 1)}{\Psi_i(\textbf{0})\Psi_i(\textbf{0}+\bbe_k)}x_k + \sum_{k\in [m]} \frac{v_{ik}\Bar{\bx}_k}{\Psi_i(\Bar{\bx})\Psi_i(\Bar{\bx}-\bbe_k)}(1-x_k) + \Phi_i(\Bar{\bx}) \right ]\label{eq:SB-SC-t-2}
\end{align}
where \eqref{eq:SB-OA-t}-\eqref{eq:SB-OA-z} are  outer-approximation cuts,  and \eqref{eq:SB-SC-z-1}-\eqref{eq:SB-SC-t-2} are supper-modular cuts.

Similar to the CP method described above, we can utilize the following two master problems Incorporated in into CP  to solve \eqref{prob-1-multicut}:
    \begin{align}\label{Bi-master-SB}\tag{\sf Bi-Master-SB}
        \min &\quad  \sum_{j \in [m]} \sum_{l\in [L]} x_j z^i_l + \sum_{l\in [L] } t_l \\
        \text{s.t.} & \quad \texttt{[Outer-approximation \& supper-modular cuts of the form \eqref{eq:SB-OA-z}-\eqref{eq:SB-SC-t-2}]}\nonumber \\
        &\quad \bx \in \cX, \by \in \bbR^n_+\nonumber
    \end{align}
and  a version where the bilinear terms are linearized:
   \begin{align}\label{Li-master-SB}\tag{\sf Li-Master-SB}
        \min &\quad \sum_{j \in [m]} \sum_{l \in [L]} \eta_{jl} + \sum_{l\in [L] } t_l \\
        \text{s.t.} & \quad \texttt{[Outer-approximation \& supper-modular cuts of the form \eqref{eq:SB-OA-z}-\eqref{eq:SB-SC-t-2}]}\nonumber \\
            & \quad \eta_{jl} \leq \left(\sum_{i\in \cG_l}\frac{\rho_i r'_i v_{ij}}{v_{i0} + v_{ij}} \right) x_j \qquad \forall l \in [L], j \in [m] \nonumber\\
    & \quad \eta_{jl} \geq \left(\sum_{i\in \cG_l} \rho_i r'_i v_{ij} \phi_{ij}(1)\right) x_j \qquad \forall l \in [L], j \in [m]\nonumber\\
    & \quad \eta_{jl} \leq z^j_l - \left(\sum_{i\in \cG_l} \rho_i r'_i v_{ij} \phi_{ij}(0)\right) (1 - x_j) \qquad \forall l \in [L], j \in [m] \nonumber\\
    & \quad \eta_{jl} \geq z^j_l - \left(\sum_{i\in \cG_l} \rho_i r'_i v_{ij} \frac{1}{v_{i0}} \right)(1 - x_j) \qquad \forall l \in [L], j \in [m] \nonumber\\
        &\quad \bx \in \cX, \pmb{\eta} \in \bbR^{m\times L}_+, \pmb{t} \in \bbR^L_+, \pmb{z}\in \bbR^{m\times L}_+\nonumber
    \end{align}
It can be seen that, at each iteration of the CP,  \(\cO(L \times m)\) cuts are added to the segment-based master problem in \eqref{Bi-master-SB} and \eqref{Li-master-SB}, while  \(\cO(n \times m)\) cuts are added to \eqref{Li-master} and \eqref{Bi-master} discussed in the previous section. As noted in previous work, the segment-based CP approach requires solving a smaller master problem at each iteration (as the number of added cuts is smaller), but the overall CP process may require more iterations to converge. This is because the segment-based approach loses some details of the objective functions, necessitating more iterations to successfully outer-approximate the nonlinear components. Choosing an appropriate number of segments helps balance the growth rate of the master problem and the total number of iterations required. In the experimental section, we will provide a detailed discussion on this matter.

\paragraph{Branch-and-Cut:}
The outer-approximation, submodular cuts, and the master problem outlined above can also be incorporated into a B\&C procedure to solve the assortment problem. Specifically, all valid cuts (outer-approximation and submodular) can be added using the \texttt{lazy-cut callback} procedure within solvers such as Gurobi and CPLEX. We also utilize two versions of the master problems, namely, bilinear and linear programs. 

To enhance the performance of the B\&C, we add Constraints \eqref{eq:relation_x_y} to the master problem. Although these constraints are redundant in the main problem formulation, they help provide tighter continuous relaxations during the B\&C procedure. It is worth noting that such constraints do not work with CP, as they directly lead to \(y_i \geq \Psi_i(\bx)\). Consequently, the master problem is no longer a restricted version of the main problem and solving the problem using CP becomes equivalent to solving the assortment problem through the MILP or CONIC reformulations.

\subsection{An Approximation Scheme}
An interesting feature of submodular maximization (or super-modular minimization) is that a simple polynomial-time greedy heuristic can guarantee a constant factor approximation of \((1 - 1/e)\). In the context of assortment optimization under the MMNL model, it is known that even without any constraints, approximating any constant factor is NP-hard. The objective in this context is not submodular, but as discussed previously, it contains some submodular components. In this section, we demonstrate that super-modularity can be leveraged to devise a simple polynomial-time greedy algorithm that yields solutions with semi-constant factor guarantees. Such solutions can be used for warm starts of CP or B\&C procedures to enhance the optimization.

First, we state in the following proposition that the objective function, defined as a set function, is monotonically increasing and supper-modular if all the prices \(r_{ij}\) are the same for all \(i \in [n]\) and \(j \in [m]\).

\begin{proposition}\label{prop:sub}
If \(r_{ij} = r_{ij'}\) for all \(i \in [n]\) and \(j,j' \in [m]\), then the objective function in \eqref{prob:Assort} is monotonically increasing and supper-modular.
\end{proposition}

We then define the following set function \(G(S)\), which is \(F(S)\) when all the prices are set to unit values:
\[
G(S) = \sum_{i \in [n]} \rho_i \left[\frac{\sum_{j \in S} v_{ij}}{v_{i0} + \sum_{j \in S} v_{ij}}\right]
\]

From Proposition \ref{prop:sub}, we know that \(G(S)\) is monotonic and supper-modular, implying that a simple greedy heuristic can offer at least a \((1 - 1/e)\) approximation solution to the problem \(\max_{S,~ |S| \leq r} G(S)\). Below, we demonstrate that this greedy procedure can also provide a guaranteed solution to the assortment optimization problem under MMNL, both in unconstrained and cardinality-constrained settings.
\begin{proposition}\label{prop:appro}
Let \(\cS = \{S_1, \ldots, S_m\}\) be a set of solutions returned by performing the Greedy Heuristic on \(\max_{|S| \leq r} G(S)\) for \(r = 1, \ldots, m\). Then, \(\cS\) contains a \(\left(\frac{r_{\min}}{r_{\max}} \left(1 - \frac{1}{e}\right)\right)\)-approximation solution to \eqref{prob:Assort}, under the cardinality constraint \(|S| \leq C\) or without any constraints. In other words
\begin{align}
\max_{S\in \cS}F(S) &\geq \frac{r_{\min}}{r_{\max}} \left(1 - \frac{1}{e}\right) \max_{S} F(S)\nonumber \\
\max_{S\in \cS, ~|S|\leq C}F(S) &\geq \frac{r_{\min}}{r_{\max}} \left(1 - \frac{1}{e}\right) \max_{|S| \leq C} F(S)    
\end{align}
where $r_{\min} = \min_{i\in [n],j\in [m]} r_{ij}$ and $r_{\max} = \max_{i\in [n],j\in [m]} r_{ij}$ 

\end{proposition}
\begin{proof}
  We first consider \eqref{prob:Assort} without a cardinality constraint \( |S| \leq C \). For any \( S \subseteq [m] \), we can see that:
\[
r_{\min} G(S) \leq F(S) \leq r_{\max} G(S)
\]
which implies 
\[
r_{\min} \max_{S,~ |S| \leq C} G(S) \leq \max_{S,~ |S| \leq C} F(S) \leq r_{\max} \max_{S,~ |S| \leq C} G(S)
\]
Moreover, since \( G(S) \) is monotonic and supper-modular, it is guaranteed that \( G(S_C) \geq (1-1/e) \max_{S,~ |S| \leq C} G(S) \). Thus we have:
\begin{align}
   r_{\max} G(S_C) &\geq \left(1-\frac{1}{e}\right) r_{\max} \max_{S,~ |S| \leq C} G(S) \nonumber\\
   &\geq \left(1-\frac{1}{e}\right) \max_{S,~ |S| \leq C} F(S) \nonumber\\
   &\geq \left(1-\frac{1}{e}\right) F(S_C) \nonumber\\
   &\geq \left(1-\frac{1}{e}\right) r_{\min} G(S_C) \nonumber
\end{align}
which implies:
\begin{align}
    \frac{F(S_C)}{\max_{S,~ |S| \leq C} F(S)} \geq \frac{\left(1-\frac{1}{e}\right) r_{\min} G(S_C)}{ r_{\max} G(S_C)} = \frac{r_{\min}}{r_{\max}} \left(1 - \frac{1}{e}\right) \label{eq:S-cS}
\end{align}
We obtain the approximation factor as desired.

For the unconstrained case, let \( S^* \) be optimal for the unconstrained problem \(\max_{S \subseteq [m]} F(S) \), and let \( C = |S^*| \). From \eqref{eq:S-cS}, we see that:
\[
F(S^*) = \max_{S,~ |S| \leq C} F(S) \leq \frac{F(S_C)}{ \frac{r_{\min}}{r_{\max}} \left(1 - \frac{1}{e}\right)}
\]
which directly implies that:
\[
\max_{S \in \cS} F(S) \geq \frac{r_{\min}}{r_{\max}} \left(1 - \frac{1}{e}\right) \max_{S \subseteq [m]} F(S)
\]
as desired. This completes the proof.
\end{proof}

Proposition \ref{prop:appro} implies that one can obtain a \(\frac{r_{\min}}{r_{\max}} \left(1 - \frac{1}{e}\right)\)-approximation solution by performing the Greedy Heuristic on \(\max_{|S| = r} G(S)\) for \(r = 1, \ldots, m\). This process can be completed in \(\mathcal{O}(m^3 n)\), where \(\mathcal{O}(mn)\) is the time required to compute the objective function \(G(S)\) for any \(S \subseteq [m]\), and \(\mathcal{O}(m^2)\) is the maximum number of steps in the Greedy Heuristic for \(r \in [m]\). To the best of our knowledge, this is the first approximation result for the assortment optimization problem under MMNL when the number of customer classes is not constant.

It is known that \eqref{prob:Assort} is NP-hard to approximate any constant factor. Our result stated in Proposition \ref{prop:appro}, however, does not conflict with this inapproximability claim, as \(r_{\min}/r_{\max}\) is not constant. In addition, the approximation in Proposition \ref{prop:appro} would be useful when the differences between prices in the set \([m]\) are not large. For instance, if \(r_{\min}/r_{\max} \geq 0.5\) (i.e., the price of the most expensive item is at most double the price of the least expensive one), then the Greedy Heuristic mentioned in Proposition \ref{prop:appro} can guarantee approximately \(0.32\) approximation solutions. Related to this approximation result, \cite{berbeglia2020assortment} provide an approximation guarantee for revenue-ordered assortments for the unconstrained mixed-logit assortment problem, and \cite{han2022fractional} establish conditions under which the assortment objective function (in the form of a subset function) is submodular. Our results are more general, as we provide approximation guarantees for both unconstrained and cardinality-constrained problems, without being restricted to additional conditions that may not hold in general settings.

To further assess the performance of the above Greedy Heuristic with respect to the performance guarantee mentioned in Proposition \ref{prop:appro}, we take instances of \{200, 500\} products and \{20, 50\} customer classes and solve them using the Greedy Heuristic. The maximum number of products in the assortment is set to \{20, 50\}, and the value \(\frac{r_{\min}}{r_{\max}}\) is generated within the range [0.25, 0.9]. Figure \ref{fig:guarantee} plots the optimal ratios given by the Greedy Heuristic, computed by taking the ratios between the objective values given by the Greedy Heuristic and the corresponding optimal values. For comparison, we also include the approximation guarantee \((1-1/e) \frac{r_{\min}}{r_{\max}}\). This figure shows that the actual optimal ratios provided by the Greedy Heuristic are significantly higher than the theoretical guarantee, which aligns with expectations. In fact, the Greedy Heuristic was able to return near-optimal solutions for the all cases considered. It can also be observed that when \(r_{\min}\) gets closer to \(r_{\max}\), the actual optimal ratio is closer to the theoretical guarantee.

\begin{figure}[!h]
    \centering
    \includegraphics[width=0.8\linewidth]{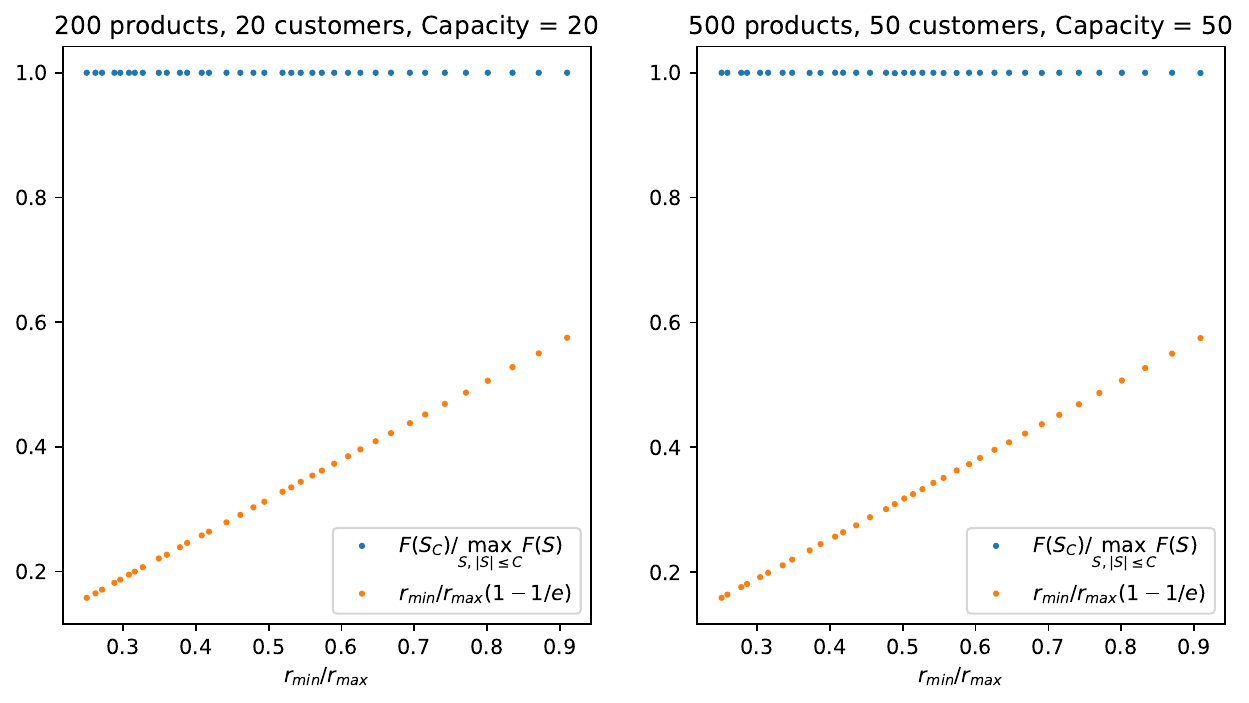}
    \caption{Optimality ratios provided by the Greedy Heuristic.}
    \label{fig:guarantee}
\end{figure}
 
\section{Numerical Study}\label{sec:result}
\subsection{Experimental settings}
To examine the effectiveness of the outer-approximation and supper-modular cuts, we perform a numerical study on four sets of problems provided by \cite{Sen2018} and several newly generated sets. The problems are solved with the Gurobi 11.0.0 solver on a computer with a 12th Gen Intel (R) CoreTM i7-1255U processor and 16 GB of RAM operating on 64-bit Windows 10. We use the default settings of Gurobi, except that we force the solver to use 8 threads to explore the solutions. All data and \texttt{.lp} files provided by \cite{Sen2018} are available at 
\url{https://atamturk.ieor.berkeley.edu/data/assortment.optimization/}. Existing \texttt{.lp} files found via the above URL are directly run with Gurobi, while our models and CONIC formulation introduced by \cite{Sen2018} for new datasets are implemented in C++. {We set a time limit of one hour for all instances. If the solver or algorithm exceeds this time limit, it is forced to stop, and the best solutions found within the allotted time are reported.}

In \cite{Sen2018}, the first and second sets of problems are randomly generated with $n = \{20, 50\}$ and $m = \{200, 500\}$, respectively. Let us denote these sets as \texttt{Sen\_200\_20} and \texttt{Sen\_500\_50}. For each product $j \in [m]$, the revenue $r_{ij}$ is the same across the customer classes $i \in [n]$ and drawn from a uniform $U[1, 3]$ distribution. The utilities $v_{ij}$ are generated from a $U[1, 2]$ distribution. The probability $\rho_i$ that the demand originates from customer class $i$ is set to $1/20$ in the first set of problems and $1/50$ in the second. The no-purchase preference $v_{i0}$ is either 5 or 10 in \texttt{Sen\_200\_20} and either 10 or 20 in \texttt{Sen\_500\_50}. The only cardinality constraint $(\sum_{j \in [m]} x_j \leq C)$ is considered, and the maximum capacity C of the assortment is taken from $\{10,20,50,100,200\}$ in \texttt{Sen\_200\_20} and $\{20,50,100,200,500\}$ in \texttt{Sen\_500\_50}. For each value of no-purchase preference and capacity, the authors generate 5 instances, resulting in a total of 50 instances.

The third set of problems in \cite{Sen2018} is associated with 100 customers and 100 products. Let us denote it as \texttt{Sen\_100\_100}. This is also the hard set of problems generated based on a working paper that is now published as \cite{Desir2022}. An undirected graph $G=(V,E)$ with $V = m = n = 100$ is created to generate the utilities $v_{ij}$. The revenue $r_{ij}$ is randomly generated from a $U[1,3]$ distribution and the probability $\rho_i$ is in $[0,1]$. The capacity $C$ is taken from $\{10,20,50,100\}$ and the no-purchase preference is either 1 or 2. There are 5 instances generated for each pair of capacity and no-purchase preference, leading to 40 instances.

The fourth set of problems is created in the same way as the first set, except that the set of 200 products is divided into 5 disjoint subsets $S_k \text{ } (k = 1,...,5)$. This leads to a general capacity constraint: $\sum_{j \in [m]} \beta_j x_j \leq \alpha$ where $\beta_j$ is randomly generated from $U[0,1]$ distribution and 5 extra capacity constraints: $\sum_{j \in [S_k]} x_j \leq C_k$. The set of pairs $(\alpha,C_k)$ is $\{(5,2),(10,4),(25,10),(50,20),(100,40)\}$ and the no-purchase preference is either 10 or 20. We denote this set of instances as \texttt{Sen\_200\_20\_5}.

Besides the four sets of instances mentioned above, we choose other parameter settings and generated  several new sets of instances as follows:
\begin{itemize}
    \item \texttt{200\_20\_10}: The utilities, the no-purchase preferences, and the revenues are the same as \texttt{Sen\_200\_20\_5}. The set of 200 products is divided into 10 disjoint subsets. The set of pairs $(\alpha,C_k)$ is $\{(5,2),(10,4),(25,10),(50,20)\}$.
    \item \texttt{500\_50\_5} and \texttt{500\_50\_10}: The utilities, the no-purchase preferences, and the revenues are generated by the same technique used for \texttt{Sen\_500\_50}. The set of 500 products is divided into 5 and 10 disjoint subsets, respectively. The $(\alpha,C_k)$ is one of $\{$(10,4), (25,10), (50,20), (125,50), (250,100)$\}$ for \texttt{500\_50\_5} and is in $\{(5,2),(10,4),(25,10),(50,20),(125,50)\}$ for \texttt{500\_50\_10}. The parameter $\beta_j$ is in $U[0,1]$.
    \item \texttt{100\_100\_5} and \texttt{100\_100\_10}: The utilities, the no-purchase preferences, and the revenues are the same as \texttt{Sen\_100\_100}. The set of 100 products is divided into 5 and 10 disjoint subsets, respectively. The set of pairs $(\alpha,C_k)$ is $\{(5,2),(10,4),(25,10),(50,20)\}$ for \texttt{100\_100\_5} and $\{(5,2),(10,4),(25,10)\}$ for \texttt{100\_100\_10}.
    \item \texttt{1000\_100}: This set of instances considers cardinality constraint only. The utilities are in $U[0,1]$ and the revenues are generated from $U[1,3]$. The probability parameters $\rho_i \text{ } (i \in [n])$ is in $U[0,1]$. The maximum capacity $(C)$ is one of $\{25,50,100,250,500\}$ and the no-purchase preference is either 10 or 20.
    \item \texttt{1000\_100\_5} and \texttt{1000\_100\_10}: The parameters of instances are the same as \texttt{1000\_100}, except that 1000 products are divided into 5 and 10 disjoint sets, and the capacity constraints for subsets are considered besides the general capacity ($\beta_j \in U[0,1] \text{ } \forall j \in [m]$). The value of each pair $(\alpha,C_k)$ for \texttt{1000\_100\_5} is chosen from the set $\{$(25,10), (50,20), (125,50), (250,100), (500,200)$\}$, and one for \texttt{1000\_100\_10} is in $\{$(10,4), (25,10), (50,20), (125,50), (250,100)$\}$.
\end{itemize}

\subsection{Comparison Results}
Our experimental results indicate that when the number of customer classes is not too large (as is the case for all the instances considered above), B\&C generally outperforms CP. Therefore, we will first focus on the comparison results for the B\&C method (with comparisons to CP provided in the appendix). In the next subsection, we will present numerical results using CP for the case of a large number of customer classes.

Table \ref{tab:capacity_only} reports comparison results for instances of single cardinality constraint. The first column lists groups of instances with the number of products varying from 100 to 1000. The second and third columns present the no-purchase preference and the maximum number of products offered in an assortment. The next five columns show number of optimal solutions found by the CONIC and B\&C algorithm under the two master programs \eqref{Bi-master} and \eqref{Li-master}. Here,  Columns ``CONIC'', ``OA'' and ``OA+SC'' indicate the numbers of optimal solutions found by the CONIC and B\&C procedure in two different scenarios, including outer-approximation cuts (OA) and combination of outer-approximation and supper-modular cut (OA+SC), within the time limit. The last five columns refer to the average running time (in seconds) for each group of instances. The symbol ``-'' means that the method cannot terminate within the time limit (3600 seconds). The highest numbers of optimal solutions found over all methods are highlighted in bold. The last row presents the total numbers of optimal solutions over all instances for each method. Here, we do not include the MILP reformulation in the comparison, as prior work \citep{Sen2018} has reported that this approach is outperformed by the CONIC approach.

\begin{table}[!h]
\centering
\resizebox{0.95\textwidth}{!}{%
\begin{tabular}{|c|c|c|c|cc|cc|r|rr|rr|}
\cline{4-13}
\multicolumn{1}{l}{} &
\multicolumn{1}{l}{} &
  \multicolumn{1}{l}{} & 
  \multicolumn{5}{|c|}{\#(Solved instances)} &
  \multicolumn{5}{c|}{Computing time (s)} \\ \hline
\multicolumn{1}{|l|}{} &
\multicolumn{1}{l|}{} &
  \multicolumn{1}{l|}{} &
  \multicolumn{1}{l}{} &
  \multicolumn{2}{|c|}{\eqref{Bi-master} + B\&C} &
  \multicolumn{2}{c|}{\eqref{Li-master} + B\&C} &
  \multicolumn{1}{l|}{} &
  \multicolumn{2}{c|}{\eqref{Bi-master} + B\&C} &
  \multicolumn{2}{c|}{\eqref{Li-master} + B\&C} \\ \cline{5-8} \cline{10-13}
  Dataset&
$v_{i0}$ &
  $\alpha$ &
  \multicolumn{1}{l|}{CONIC} &
  \multicolumn{1}{l}{OA} &
  \multicolumn{1}{l|}{OA+SC} &
  \multicolumn{1}{l}{OA} &
  \multicolumn{1}{l|}{OA+SC} &
  CONIC &
  OA &
  OA+SC &
  OA &
  OA+SC \\ \hline
\multirow{10}{*}{\texttt{Sen\_200\_20}}  &
\multirow{5}{*}{5}  & 10  & \textbf{5} & \textbf{5} & \textbf{5} & \textbf{5} & \textbf{5} & 2.54 & \textbf{2.40} & 2.82 & 4.44 & 5.01 \\
&& 20  & \textbf{5} & \textbf{5} & \textbf{5} & \textbf{5} & \textbf{5} & 3.07 & \textbf{2.32} & 3.19 & 3.12 & 3.49 \\
&& 50  & \textbf{5} & \textbf{5} & \textbf{5} & \textbf{5} & \textbf{5} & 1.57 & 1.25 & 1.50 & \textbf{1.13} & 1.23 \\
&& 100 & \textbf{5} & \textbf{5} & \textbf{5} & \textbf{5} & \textbf{5} & 0.68 & 0.96 & 1.12 & \textbf{0.53} & 0.58 \\
&& 200 & \textbf{5} & \textbf{5} & \textbf{5} & \textbf{5} & \textbf{5} & 0.74 & 1.03 & 1.29 & \textbf{0.56} & \textbf{0.56} \\ \cline{2-13}
&\multirow{5}{*}{10} & 10  & \textbf{5} & \textbf{5} & \textbf{5} & \textbf{5} & \textbf{5} & 2.95 & \textbf{1.47} & 1.83 & 4.11 & 4.24 \\
&& 20  & \textbf{5} & \textbf{5} & \textbf{5} & \textbf{5} & \textbf{5} & 3.06 & \textbf{1.70} & 2.47 & 3.25 & 3.25 \\
&& 50  & \textbf{5} & \textbf{5} & \textbf{5} & \textbf{5} & \textbf{5} & 4.67 & \textbf{1.68} & 2.04 & 1.96 & 2.04 \\
&& 100 & \textbf{5} & \textbf{5} & \textbf{5} & \textbf{5} & \textbf{5} & 0.79 & 0.97 & 1.20 & \textbf{0.54} & 0.57 \\
&& 200 & \textbf{5} & \textbf{5} & \textbf{5} & \textbf{5} & \textbf{5} & 0.65 & 1.00 & 1.26 & 0.70 & \textbf{0.69} \\ \hline
\multirow{10}{*}{\texttt{Sen\_500\_50}}  &
\multirow{5}{*}{10}  & 20  & \textbf{5} & \textbf{5} & \textbf{5} & \textbf{5} & \textbf{5} & 99.32  & 40.08 & \textbf{38.26} & 127.63 & 126.22 \\
&& 50  & \textbf{5} & \textbf{5} & \textbf{5} & \textbf{5} & \textbf{5} & 60.61  & 59.33 & \textbf{58.96} & 63.51  & 63.86  \\
&& 100 & \textbf{5} & \textbf{5} & \textbf{5} & \textbf{5} & \textbf{5} & 36.68  & 32.99 & 38.43 & \textbf{27.15}  & 28.19  \\
&& 200 & \textbf{5} & \textbf{5} & \textbf{5} & \textbf{5} & \textbf{5} & 22.13  & 14.71 & 17.51 & 8.11   & \textbf{7.29}   \\
&& 500 & \textbf{5} & \textbf{5} & \textbf{5} & \textbf{5} & \textbf{5} & 25.50  & 21.58 & 20.70 & \textbf{9.43}   & 9.45   \\ \cline{2-13}
&\multirow{5}{*}{20} & 20  & \textbf{5} & \textbf{5} & \textbf{5} & \textbf{5} & \textbf{5} & 111.36 & 23.48 & \textbf{17.42} & 150.77 & 156.31 \\
&& 50  & \textbf{5} & \textbf{5} & \textbf{5} & \textbf{5} & \textbf{5} & 109.77 & \textbf{84.90} & 90.43 & 108.91 & 108.95 \\
&& 100 & \textbf{5} & \textbf{5} & \textbf{5} & \textbf{5} & \textbf{5} & 70.30  & 70.95 & 73.74 & \textbf{65.37}  & 70.93  \\
&& 200 & \textbf{5} & \textbf{5} & \textbf{5} & \textbf{5} & \textbf{5} & 22.30  & 20.63 & 20.24 & 8.69   & \textbf{8.19}   \\
&& 500 & \textbf{5} & \textbf{5} & \textbf{5} & \textbf{5} & \textbf{5} & 26.05  & 20.17 & 18.86 & 9.11   & \textbf{8.28}   \\ \hline
\multirow{8}{*}{\texttt{Sen\_100\_100}}&
\multirow{4}{*}{1} & 10    & \textbf{5} & \textbf{5} & \textbf{5} & \textbf{5} & \textbf{5} & 4.29 & \textbf{4.14} & 4.43 & 30.36 & 27.72 \\
&& 20    & \textbf{5} & \textbf{5} & \textbf{5} & \textbf{5} & \textbf{5} & 8.66 & 4.78 & \textbf{4.64} & 50.95 & 40.17 \\
&& 50    & \textbf{5} & \textbf{5} & \textbf{5} & \textbf{5} & \textbf{5} & 1.52 & 1.02 & \textbf{0.99} & 4.31  & 4.41  \\
&& 100 & \textbf{5} & \textbf{5} & \textbf{5} & \textbf{5} & \textbf{5} & \textbf{0.38} & 0.41 & 0.43 & 2.22  & 2.32  \\ \cline{2-13}
&\multirow{4}{*}{2} & 10    & \textbf{5} & \textbf{5} & \textbf{5} & \textbf{5} & \textbf{5} & 2.49 & \textbf{2.42} & 2.51 & 16.55 & 15.53 \\
&& 20    & \textbf{5} & \textbf{5} & \textbf{5} & \textbf{5} & \textbf{5} & 6.59 & \textbf{3.77} & 3.91 & 27.95 & 25.96 \\
&& 50    & \textbf{5} & \textbf{5} & \textbf{5} & \textbf{5} & \textbf{5} & 3.64 & \textbf{1.59} & 2.03 & 8.24  & 9.52  \\
&& 100 & \textbf{5} & \textbf{5} & \textbf{5} & \textbf{5} & \textbf{5} & \textbf{0.21} & 0.22 & 0.24 & 1.03  & 1.10  \\ \hline

\multirow{10}{*}{\texttt{1000\_100}}&
\multirow{5}{*}{10} & 25  & 1 & \textbf{5} &\textbf{5}  & \textbf{5} &\textbf{5}  & 3530.94 & \textbf{288.27} &293.99  & 539.53 &1064.15  \\
&& 50  & \textbf{5} & \textbf{5} & \textbf{5} & \textbf{5} & \textbf{5} & 1637.69 & 450.78 & \textbf{397.34} & 685.28 & 518.85 \\
&& 100 & \textbf{5} & \textbf{5} & \textbf{5} & \textbf{5} & \textbf{5} & 439.77  & 305.85 & \textbf{226.62} & 402.85 & 290.64 \\
&& 250 & \textbf{5} & \textbf{5} & \textbf{5} & \textbf{5} & \textbf{5} & 104.67  & 74.23  & 67.70 & 43.33  & \textbf{38.84} \\
&& 500 & \textbf{5} & \textbf{5} & \textbf{5} & \textbf{5} & \textbf{5} & 135.42  & 80.27  & 83.84 & \textbf{52.73}  & 56.75 \\ \cline{2-13}
&\multirow{5}{*}{20} & 25  & 0 & \textbf{5} & \textbf{5} & \textbf{5} & \textbf{5} & -       & 149.04 & \textbf{132.68} & 175.08 & 166.27 \\
&& 50  & \textbf{5} & \textbf{5} & \textbf{5} & \textbf{5} & \textbf{5} & 1467.83 & 385.68 & \textbf{283.00} & 603.74 & 450.73 \\
&& 100 & \textbf{5} & \textbf{5} & \textbf{5} & \textbf{5} & \textbf{5} & 1368.80 & 400.76 & \textbf{284.13} & 490.18 & 443.51 \\
&& 250 & \textbf{5} & \textbf{5} & \textbf{5} & \textbf{5} & \textbf{5} & 348.09  & 76.41  & 72.88 & \textbf{50.43}  &  50.45\\
&& 500 & \textbf{5} & \textbf{5} & \textbf{5} & \textbf{5} & \textbf{5} & 107.90  & 70.29  & 72.85 & 34.65  &  \textbf{34.46}\\ \hline
\multicolumn{3}{|r|}{Summary} & 181 & \textbf{190} & \textbf{190} & \textbf{190} & \textbf{190} & 191.96  & 71.15  & \textbf{61.78} & 100.49  &  101.34\\ \hline
\end{tabular}%
}
\caption{Results for \texttt{Sen\_200\_20}, \texttt{Sen\_500\_50}, \texttt{Sen\_100\_100} and \texttt{1000\_100}. Time limit set to one hour. Average values are calculated by taking into account only those instances solved to optimality by the respective approach. Best values (the largest number of instances solved, or the lowest computing time) are highlighted in bold.}
\label{tab:capacity_only}
\end{table}

As shown in Table \ref{tab:capacity_only}, the B\&C methods are able to solve all the instances from \cite{Sen2018} to optimality, whereas the CONIC approach provides optimal solutions for 181 out of 190 instances. Specifically, for the largest-sized instances from the \texttt{1000\_100} dataset, the B\&C methods find all optimal solutions out of 50 instances compared to 41 for the CONIC approach. 
In terms of computing time, the B\&C approaches are approximately 1.89 to 3.1 times faster than CONIC for all instances solved optimally. Among the B\&C methods, the algorithm based on \eqref{Bi-master} with OA+SC appears to be the fastest.

\begin{table}[!h]
\centering
\resizebox{0.95\textwidth}{!}{%
\begin{tabular}{|c|c|c|c|cc|cc|r|rr|rr|}
\cline{4-13}
\multicolumn{1}{l}{} &
\multicolumn{1}{l}{} &
  \multicolumn{1}{l}{} & 
  \multicolumn{5}{|c|}{\#(Solved instances)} &
  \multicolumn{5}{c|}{Computing time (s)} \\ \hline
\multicolumn{1}{|l|}{} &
\multicolumn{1}{l|}{} &
  \multicolumn{1}{l|}{} &
  \multicolumn{1}{l}{} &
  \multicolumn{2}{|c|}{$\eqref{Bi-master}$ + B\&C} &
  \multicolumn{2}{c|}{$\eqref{Li-master}$ + B\&C} &
  \multicolumn{1}{l|}{} &
  \multicolumn{2}{c|}{$\eqref{Bi-master}$ + B\&C} &
  \multicolumn{2}{c|}{$\eqref{Li-master}$ + B\&C} \\ \cline{5-8} \cline{10-13}
  $k$&
$v_{i0}$ &
  $\alpha,\alpha_k$ &
  \multicolumn{1}{l|}{CONIC} &
  \multicolumn{1}{l}{OA} &
  \multicolumn{1}{l|}{OA+SC} &
  \multicolumn{1}{l}{OA} &
  \multicolumn{1}{l|}{OA+SC} &
  CONIC &
  OA &
  OA+SC &
  OA &
  OA+SC \\ \hline
\multirow{10}{*}{5}&
\multirow{5}{*}{10} 
& 5,2    & \textbf{5} & \textbf{5} & \textbf{5} & \textbf{5} & \textbf{5} & 3.99  & \textbf{2.38} & 2.44 & 3.47 & 3.65 \\
&& 10,4   & \textbf{5} & \textbf{5} & \textbf{5} & \textbf{5} & \textbf{5} & 7.95  & \textbf{4.34} & 4.60 & 4.55 & 4.81 \\
&& 25,10  & \textbf{5} & \textbf{5} & \textbf{5} & \textbf{5} & \textbf{5} & 22.99 & 8.45 & 8.02 & \textbf{6.43} & 7.18 \\
&& 50,20  & \textbf{5} & \textbf{5} & \textbf{5} & \textbf{5} & \textbf{5} & 1.20  & 1.13 & 1.11 & \textbf{1.02} & 1.05 \\
&& 100,40 & \textbf{5} & \textbf{5} & \textbf{5} & \textbf{5} & \textbf{5} & 0.86  & 0.95 & 0.98 & \textbf{0.80} & 0.81 \\ \cline{2-13}
&\multirow{5}{*}{20} 
& 5,2    & \textbf{5} & \textbf{5} & \textbf{5} & \textbf{5} & \textbf{5} & 3.76  & 1.84 & \textbf{1.79} & 2.49 & 2.30 \\
&& 10,4   & \textbf{5} & \textbf{5} & \textbf{5} & \textbf{5} & \textbf{5} & 5.70  & 3.31 & 3.30 & \textbf{3.09} & 3.26 \\
&& 25,10  & \textbf{5} & \textbf{5} & \textbf{5} & \textbf{5} & \textbf{5} & 25.98 & 7.22 & 8.89 & \textbf{6.26} & 6.78 \\
&& 50,20  & \textbf{5} & \textbf{5} & \textbf{5} & \textbf{5} & \textbf{5} & 8.52  & 2.07 & \textbf{1.97} & 2.06 & 2.14 \\
&& 100,40 & \textbf{5} & \textbf{5} & \textbf{5} & \textbf{5} & \textbf{5} & 0.71  & 1.14 & 1.19 & \textbf{0.70} & 0.71 \\ \hline
\multirow{8}{*}{10}&
\multirow{4}{*}{10} 
& 5,2     & \textbf{5} & \textbf{5} & \textbf{5} & \textbf{5} & \textbf{5} & 18.24  & 9.43   & 10.07  & \textbf{7.24}  & 8.73  \\
&& 10,4    & \textbf{5} & \textbf{5} & \textbf{5} & \textbf{5} & \textbf{5} & 122.25 & 102.03 & 98.63  & \textbf{31.95} & 33.13 \\
&& 25,10   & \textbf{5} & \textbf{5} & \textbf{5} & \textbf{5} & \textbf{5} & 64.34  & 76.39  & 68.01  & \textbf{57.04} & 65.42 \\
&& 50,20 & \textbf{5} & \textbf{5} & \textbf{5} & \textbf{5} & \textbf{5} & 0.91   & 1.00   & 0.98   & 0.88  & \textbf{0.85}  \\ \cline{2-13}
&\multirow{4}{*}{20} 
& 5,2     & \textbf{5} & \textbf{5} & \textbf{5} & \textbf{5} & \textbf{5} & 10.77  & 6.73   & 6.50   & 5.16  & \textbf{5.00}  \\
&& 10,4    & \textbf{5} & \textbf{5} & \textbf{5} & \textbf{5} & \textbf{5} & 46.29  & 17.41  & 19.56  & 15.83 & \textbf{15.82} \\
&& 25,10   & \textbf{5} & \textbf{5} & \textbf{5} & \textbf{5} & \textbf{5} & 162.59 & 122.12 & 101.43 & \textbf{70.46} & 99.31 \\
&& 50,20 & \textbf{5} & \textbf{5} & \textbf{5} & \textbf{5} & \textbf{5} & 7.21   & 3.03   & 2.87   & \textbf{2.54}  & 2.60  \\ \hline
\multicolumn{3}{|r|}{Summary} & \textbf{90} & \textbf{90} & \textbf{90} & \textbf{90} & \textbf{90} & 28.57   & 20.61   & 19.02   & \textbf{12.78}  & 14.56  \\ \hline
\end{tabular}%
}
\caption{Results for $\texttt{Sen\_200\_20\_5}$ and $\texttt{200\_20\_10}$}
\label{tab:200_20_k}
\end{table}

Tables \ref{tab:200_20_k} and \ref{tab:100_10_k} present the comparison results of the CONIC and B\&C methods on the \texttt{100\_100\_5}, \texttt{100\_100\_10}, \texttt{Sen\_200\_20\_5}, and \texttt{200\_20\_10} datasets. Both methods perform well on these datasets when all instances are solved to optimality. Specifically, for the \texttt{Sen\_200\_20\_5} and \texttt{200\_20\_10} datasets, the B\&C method using OA within its procedure is the fastest for 11 out of 18 instance sets. On average, this approach is about 2.23 times faster than the highest runtime, which comes from the CONIC approach, across all instances. For the \texttt{100\_100\_5} and \texttt{100\_100\_10} datasets, the B\&C method with \eqref{Bi-master} and OA is the most efficient in terms of average runtime, followed by B\&C with \eqref{Bi-master} and OA+SC. Overall, the B\&C method with \eqref{Bi-master} and OA cuts is approximately 9 to 15.89 times faster than both the CONIC approach and the B\&C method with \eqref{Li-master}, achieving the best average computation times for 15 out of 18 instance sets.

\begin{table}[htb]
\centering
\resizebox{0.95\textwidth}{!}{%
\begin{tabular}{|c|c|c|c|cc|cc|r|rr|rr|}
\cline{4-13}
\multicolumn{1}{l}{} &
\multicolumn{1}{l}{} &
  \multicolumn{1}{l}{} & 
  \multicolumn{5}{|c|}{\#(Solved instances)} &
  \multicolumn{5}{c|}{Computing time (s)} \\ \hline
\multicolumn{1}{|l|}{} &
\multicolumn{1}{l|}{} &
  \multicolumn{1}{l|}{} &
  \multicolumn{1}{l}{} &
  \multicolumn{2}{|c|}{$\eqref{Bi-master}$ + B\&C} &
  \multicolumn{2}{c|}{$\eqref{Li-master}$ + B\&C} &
  \multicolumn{1}{l|}{} &
  \multicolumn{2}{c|}{$\eqref{Bi-master}$ + B\&C} &
  \multicolumn{2}{c|}{$\eqref{Li-master}$ + B\&C} \\ \cline{5-8} \cline{10-13}
  $k$&
$v_{i0}$ &
  $\alpha,\alpha_k$ &
  \multicolumn{1}{l|}{CONIC} &
  \multicolumn{1}{l}{OA} &
  \multicolumn{1}{l|}{OA+SC} &
  \multicolumn{1}{l}{OA} &
  \multicolumn{1}{l|}{OA+SC} &
  CONIC &
  OA &
  OA+SC &
  OA &
  OA+SC \\ \hline
  \multirow{8}{*}{5}&
\multirow{4}{*}{1} 
& 5,2   & \textbf{5} & \textbf{5} & \textbf{5} & \textbf{5} & \textbf{5} & 58.33  & \textbf{3.78}  & 4.04  & 47.46  & 42.86  \\
&& 10,4  & \textbf{5} & \textbf{5} & \textbf{5} & \textbf{5} & \textbf{5} & 177.19 & \textbf{13.18} & 15.09 & 135.60 & 112.90 \\
&& 25,10 & \textbf{5} & \textbf{5} & \textbf{5} & \textbf{5} & \textbf{5} & 46.11  & \textbf{1.47}  & 1.69  & 11.15  & 11.03  \\
&& 50,20 & \textbf{5} & \textbf{5} & \textbf{5} & \textbf{5} & \textbf{5} & 3.79   & \textbf{0.35}  & 0.38  & 2.11   & 2.26   \\ \cline{2-13}
&\multirow{4}{*}{2} 
& 5,2   & \textbf{5} & \textbf{5} & \textbf{5} & \textbf{5} & \textbf{5} & 36.31  & \textbf{3.25}  & 3.54  & 29.93  & 26.50  \\
&& 10,4  & \textbf{5} & \textbf{5} & \textbf{5} & \textbf{5} & \textbf{5} & 124.59 & \textbf{8.53}  & 8.96  & 69.41  & 95.93  \\
&& 25,10 & \textbf{5} & \textbf{5} & \textbf{5} & \textbf{5} & \textbf{5} & 97.72  & \textbf{2.46}  & 2.57  & 23.29  & 24.07  \\
&& 50,20 & \textbf{5} & \textbf{5} & \textbf{5} & \textbf{5} & \textbf{5} & 1.57   & \textbf{0.21}  & 0.23  & 1.03   & 1.03   \\ \hline
 \multirow{6}{*}{10}&
\multirow{3}{*}{1} 
& 5,2   & \textbf{5} & \textbf{5} & \textbf{5} & \textbf{5} & \textbf{5} & 119.58 & \textbf{8.62} & \textbf{8.62} & 88.73 & 84.48 \\
&& 10,4  & \textbf{5} & \textbf{5} & \textbf{5} & \textbf{5} & \textbf{5} & 96.16  & \textbf{5.41} & 5.67 & 49.49 & 45.46 \\
&& 25,10 & \textbf{5} & \textbf{5} & \textbf{5} & \textbf{5} & \textbf{5} & 20.15  & 1.56 & \textbf{1.38} & 6.50  & 6.93  \\ \cline{2-13}
&\multirow{3}{*}{2} 
& 5,2   & \textbf{5} & \textbf{5} & \textbf{5} & \textbf{5} & \textbf{5} & 66.69  & \textbf{5.61} & 6.70 & 44.66 & 47.83 \\
&& 10,4  & \textbf{5} & \textbf{5} & \textbf{5} & \textbf{5} & \textbf{5} & 79.60  & \textbf{4.10} & 4.55 & 33.03 & 35.62 \\
&& 25,10 & \textbf{5} & \textbf{5} & \textbf{5} & \textbf{5} & \textbf{5} & 35.70  & 2.07 & \textbf{1.82} & 10.76 & 10.49 \\ \hline
\multicolumn{3}{|r|}{Summary} & \textbf{70} & \textbf{70} & \textbf{70} & \textbf{70} & \textbf{70} & 68.82   & \textbf{4.33}   & 4.66   & 39.51  & 39.10  \\ \hline
\end{tabular}%
}
\caption{Results for $\texttt{100\_100\_5}$ and $\texttt{100\_100\_10}$}
\label{tab:100_10_k}
\end{table}

\begin{table}[htb]
\centering
\resizebox{0.95\textwidth}{!}{%
\begin{tabular}{|c|c|c|c|cc|cc|r|rr|rr|}
\cline{4-13}
\multicolumn{1}{l}{} &
\multicolumn{1}{l}{} &
  \multicolumn{1}{l}{} & 
  \multicolumn{5}{|c|}{\#(Solved instances)} &
  \multicolumn{5}{c|}{Computing time (s)} \\ \hline
\multicolumn{1}{|l|}{} &
\multicolumn{1}{l|}{} &
  \multicolumn{1}{l|}{} &
  \multicolumn{1}{l}{} &
  \multicolumn{2}{|c|}{$\eqref{Bi-master}$ + B\&C} &
  \multicolumn{2}{c|}{$\eqref{Li-master}$ + B\&C} &
  \multicolumn{1}{l|}{} &
  \multicolumn{2}{c|}{$\eqref{Bi-master}$ + B\&C} &
  \multicolumn{2}{c|}{$\eqref{Li-master}$ + B\&C} \\ \cline{5-8} \cline{10-13}
  $k$&
$v_{i0}$ &
  $\alpha,\alpha_k$ &
  \multicolumn{1}{l|}{CONIC} &
  \multicolumn{1}{l}{OA} &
  \multicolumn{1}{l|}{OA+SC} &
  \multicolumn{1}{l}{OA} &
  \multicolumn{1}{l|}{OA+SC} &
  CONIC &
  OA &
  OA+SC &
  OA &
  OA+SC \\ \hline
\multirow{10}{*}{5}&
\multirow{5}{*}{10} 
& 10,4    & \textbf{5} & \textbf{5} & \textbf{5} & \textbf{5} & \textbf{5} & 170.94 & 111.42 & \textbf{98.78}  & 135.39 & 138.71 \\
&& 25,10   & \textbf{5} & \textbf{5} & \textbf{5} & \textbf{5} & \textbf{5} & 445.96 & \textbf{219.92} & 227.15 & 237.33 & 229.27 \\
&& 50,20   & \textbf{5} & \textbf{5} & \textbf{5} & \textbf{5} & \textbf{5} & 121.58 & 86.02  & 83.81  & \textbf{80.99}  & 87.07  \\
&& 125,50  & \textbf{5} & \textbf{5} & \textbf{5} & \textbf{5} & \textbf{5} & 19.63  & 16.71  & 17.32  & 11.38  & \textbf{9.16}   \\
&& 250,100 & \textbf{5} & \textbf{5} & \textbf{5} & \textbf{5} & \textbf{5} & 25.42  & 15.87  & 15.81  & 9.50   & \textbf{9.27}   \\ \cline{2-13}
&\multirow{5}{*}{20} 
& 10,4    & \textbf{5} & \textbf{5} & \textbf{5} & \textbf{5} & \textbf{5} & 234.75 & 87.50  & 94.49  & \textbf{69.95}  & 76.60  \\
&& 25,10   & \textbf{5} & \textbf{5} & \textbf{5} & \textbf{5} & \textbf{5} & 253.45 & 138.81 & \textbf{133.55} & 147.29 & 148.65 \\
&& 50,20   & \textbf{5} & \textbf{5} & \textbf{5} & \textbf{5} & \textbf{5} & 164.60 & \textbf{102.88} & 113.25 & 111.89 & 122.49 \\
&& 125,50  & \textbf{5} & \textbf{5} & \textbf{5} & \textbf{5} & \textbf{5} & 19.64  & 19.84  & 21.33  & \textbf{7.11}   & 9.27   \\
&& 250,100 & \textbf{5} & \textbf{5} & \textbf{5} & \textbf{5} & \textbf{5} & 21.25  & 27.57  & 27.56  & \textbf{8.51}   & \textbf{8.51}   \\ \hline

\multirow{10}{*}{10}&
\multirow{5}{*}{10} 
& 5,2    & \textbf{5} & \textbf{5} & \textbf{5} & \textbf{5} & \textbf{5} & 213.13  & 146.71  & \textbf{141.02} & 169.10  & 173.70  \\
&& 10,4   & \textbf{5} & \textbf{5} & \textbf{5} & \textbf{5} & \textbf{5} & 964.36  & 1590.73 & 1800.49 & \textbf{381.67}  & 410.77  \\
&& 25,10  & 0 & \textbf{1} & \textbf{1} & 0 & 0  & -       &      3278.09   & \textbf{3120.50} & -       &    -     \\
&& 50,20  & 1 & 4 & \textbf{5} & \textbf{5} & \textbf{5} & 3078.21 & 1181.92 & 1230.10 & \textbf{982.29}  & 1559.07    \\
&& 125,50 & \textbf{5} & \textbf{5} & \textbf{5} & \textbf{5} & \textbf{5} & 21.17   & 18.25   & 25.60 & \textbf{9.95}    & 10.28    \\ \cline{2-13}
&\multirow{5}{*}{20} 
& 5,2    & \textbf{5} & \textbf{5} & \textbf{5} & \textbf{5} & \textbf{5} & 193.94  & 91.79   & \textbf{85.95} & 94.17   & 90.91   \\
&& 10,4   & \textbf{5} & \textbf{5} & \textbf{5} & \textbf{5} & \textbf{5} & 450.39  & 306.11  & 309.66 & \textbf{242.82}  & 269.76  \\
&& 25,10  & 0 & \textbf{1}  & \textbf{1} & 0 &  0 & -       &      \textbf{3470.64}   & 3474.20 & -       &   -      \\
&& 50,20  & 0 & \textbf{3} & 2 & 1 & 1 & -       & \textbf{2719.53} & 2857.56 & 2933.72 & 3317.22 \\
&& 125,50 & \textbf{5} & \textbf{5} &  \textbf{5}& \textbf{5} & \textbf{5} & 15.23   & 26.61   & 22.30 & \textbf{6.68}    & 7.04    \\ \hline
\multicolumn{3}{|r|}{Summary} & 81 & \textbf{89} & \textbf{89} & 86 & 86 & 243.89   & 389.54   & 388.21   & \textbf{191.44}  & 233.95  \\ \hline
\end{tabular}%
}
\caption{Results for $\texttt{500\_50\_5}$ and $\texttt{500\_50\_10}$}
\label{tab:500_50_k}
\end{table}

\begin{table}[htb]
\centering
\resizebox{0.95\textwidth}{!}{%
\begin{tabular}{|c|c|c|c|cc|cc|r|rr|rr|}
\cline{4-13}
\multicolumn{1}{l}{} &
\multicolumn{1}{l}{} &
  \multicolumn{1}{l}{} & 
  \multicolumn{5}{|c|}{\#(Solved instances)} &
  \multicolumn{5}{c|}{Computing time (s)} \\ \hline
\multicolumn{1}{|l|}{} &
\multicolumn{1}{l|}{} &
  \multicolumn{1}{l|}{} &
  \multicolumn{1}{l}{} &
  \multicolumn{2}{|c|}{$\eqref{Bi-master}$ + B\&C} &
  \multicolumn{2}{c|}{$\eqref{Li-master}$ + B\&C} &
  \multicolumn{1}{l|}{} &
  \multicolumn{2}{c|}{$\eqref{Bi-master}$ + B\&C} &
  \multicolumn{2}{c|}{$\eqref{Li-master}$ + B\&C} \\ \cline{5-8} \cline{10-13}
  $k$&
$v_{i0}$ &
  $\alpha,\alpha_k$ &
  \multicolumn{1}{l|}{CONIC} &
  \multicolumn{1}{l}{OA} &
  \multicolumn{1}{l|}{OA+SC} &
  \multicolumn{1}{l}{OA} &
  \multicolumn{1}{l|}{OA+SC} &
  CONIC &
  OA &
  OA+SC &
  OA &
  OA+SC \\ \hline
\multirow{10}{*}{5}&
\multirow{5}{*}{10} & 25,10   & 2 & 3 & 3 & \textbf{4} & \textbf{4} & 1974.92 & 1291.77 & 1438.75 & 1136.26 & \textbf{994.03} \\
&& 50,20   & 3 & 3 & \textbf{4} & \textbf{4} & \textbf{4} & 1011.68 & 609.82 & 1067.69 & 661.68  & \textbf{542.27} \\
&& 125,50  & 4 & \textbf{5} & \textbf{5} & \textbf{5} & \textbf{5} & 122.70  & 143.86 & 143.76 & 125.04  & \textbf{108.85} \\
&& 250,100 & \textbf{5} & \textbf{5} & \textbf{5} & \textbf{5} & \textbf{5} & 134.52  & 106.68 & 110.95 & \textbf{68.34}   & 68.37 \\
&& 500,200 & \textbf{5} & \textbf{5} & \textbf{5} & \textbf{5} & \textbf{5} & 170.70  & 94.35 & 95.45 & \textbf{73.05}   & 81.12 \\ \cline{2-13}
&\multirow{5}{*}{20} & 25,10   & 3 & \textbf{5} & \textbf{5} & \textbf{5} & \textbf{5} & 1542.77 & 1280.33 & 1197.57 & 739.52  & \textbf{576.33} \\
&& 50,20   & 2 & 3 & 3 & \textbf{4} & \textbf{4} & 1750.03 & 1155.64 &  866.86 & 878.51  & \textbf{781.73} \\
&& 125,50  & 3 & 4 & 3 & \textbf{5} & \textbf{5} & 889.12  & 1145.85 & \textbf{373.08} & 925.43  & 826.23  \\
&& 250,100 & \textbf{5} & \textbf{5} & \textbf{5} & \textbf{5} & \textbf{5} & 217.07  & 95.40 & 102.95 & \textbf{43.43}   & 43.65 \\
&& 500,200 & \textbf{5} & \textbf{5} & \textbf{5} & \textbf{5} & \textbf{5} & 110.65  & 87.82 & 94.28 & \textbf{45.24}   & 46.13 \\ \hline

\multirow{10}{*}{10}&
\multirow{5}{*}{10} & 10,4    & 1 & 3 & 3 & \textbf{4} & \textbf{4} & 3400.84 & 2475.91 & 2343.05 & \textbf{1481.40} & 1563.57 \\
&& 25,10   & 0 & 0 & 0 & 0 & 0 & -       & -       & - & -       & - \\
&& 50,20   & 0 & 0 & 0 & 0 & 0 & -  & -    & - & -  & - \\
&& 125,50  & \textbf{5} & \textbf{5} & \textbf{5} & \textbf{5} & \textbf{5} & 127.51  & 106.98  & 105.00 & 68.34   & \textbf{57.38} \\
&& 250,100 & \textbf{5} & \textbf{5} & \textbf{5} & \textbf{5} & \textbf{5} & 194.29  & 92.33   & 89.66 & \textbf{73.05}   & 76.22 \\ \cline{2-13}
&\multirow{5}{*}{20} & 10,4    & 3 & \textbf{5} & \textbf{5} & \textbf{5} & \textbf{5} & 2593.68 & 2115.63 & 2177.90 & 1034.69 & \textbf{940.20} \\
&& 25,10   & 0 & 0 & 0 & 0 & 0 & -       & -       & - & -       & - \\
&& 50,20   & 0 & 0 & 0 & 0 & 0 & -       & -       & - & -       & - \\
&& 125,50  & 0 & 0 & 0 & \textbf{5} & \textbf{5} & -       & -       & - & \textbf{1433.94} & 1565.53 \\
&& 250,100 & \textbf{5} & \textbf{5} & \textbf{5} & \textbf{5} & \textbf{5} & 110.65  & 92.40   & 105.41 & \textbf{45.38}   & 48.81 \\ \hline
\multicolumn{3}{|r|}{Summary} & 56 & 66 & 66 & \textbf{76} & \textbf{76} &  620.64  &  640.33 &  612.89  & \textbf{489.61}  &  496.32 \\ \hline
\end{tabular}%
}
\caption{Results for $\texttt{1000\_100\_5}$ and $\texttt{1000\_100\_10}$}
\label{tab:1000_100_k}
\end{table}

We also present comparison results for larger instances (i.e., larger numbers of customer classes and products) in Tables \ref{tab:500_50_k} and \ref{tab:1000_100_k}. These results are consistent with earlier experiments. The B\&C approaches outperform the CONIC method in terms of both the number of optimal solutions found and computing time. For the 500-product dataset, the best approach is \eqref{Bi-master} with 89 out of 100 optimal solutions found, followed by \eqref{Li-master} with 86 out of 100, and the CONIC method with 81 out of 100. The two fastest settings of the B\&C algorithm are those based on the linear master problem \eqref{Li-master}. Although the B\&C method based on \eqref{Bi-master} requires higher runtimes compared to those based on the linear master problem, it is important to note that the computation time is calculated based on instances solved to optimality, and the B\&C with \eqref{Bi-master} solves the most instances, including difficult and time-consuming ones. In Table \ref{tab:1000_100_k}, the B\&C based on \eqref{Li-master} significantly outperforms other methods by finding 76 out of 100 optimal solutions, compared to 66 out of 100 for the second-best formulation \eqref{prob:BiCP} and 56 out of 100 for the CONIC method. Furthermore, the \eqref{Li-master} with OA and OA+SC embedded in the B\&C procedure are the two fastest approaches. The average runtime of the OA-based version is better than the other, and both are about 20\% lower than that of the CONIC method.

\subsection{Experiment Results for Instances of Large Numbers of Customer Classes}
In this section, we present comparison results for large-scale instances involving a large number of customer classes. Such instances are particularly relevant when, for example, the choice probabilities under the mixed-logit model are expressed as expectations over random variables. In these cases, sample average approximation (SAA) may be needed to approximate the objective function, requiring a large number of samples to achieve the desired accuracy.  To evaluate the performance of our approaches on these large-scale instances, we generated four new instance sets, named \texttt{100\_1000}, \texttt{200\_2000}, \texttt{200\_4000}, and \texttt{100\_5000}. These sets contain individual customer numbers ranging from 1000 to 5000, and the number of products is either 100 or 200. These instances contain general capacity constraints of the form \(\sum_{j \in [m]} \beta_j x_j \leq \alpha\), where \(\beta_j\) is randomly generated from a \(U[0,1]\) distribution, and \(\alpha\) is chosen from the set \(\{10, 20, 50\}\). The utilities are drawn from a \(U[0,1]\) distribution, and the revenues are generated from a \(U[1,3]\) distribution.


As discussed earlier, when the number of customer classes is large, adding cuts for each customer class, as in \eqref{Bi-master} and \eqref{Li-master}, can cause the master problem to grow quickly, resulting in long computing times. The segment-based approach discussed in Section [] addresses this by dividing the entire set of customer classes into disjoint subsets and adding cuts for each group of customers. As shown in \cite{MaiLodi2020_OA}, increasing the number of customer groups makes the master problem grow faster in terms of the number of constraints, leading to longer solving times. However, it also helps reduce the number of iterations needed for convergence. Therefore, it is crucial to select an appropriate group size to balance the growth rate of the master problem and the number of iterations required for convergence. In our experiments, we chose 20 customer groups, as this provided the best overall performance (see  Figure \ref{fig:multicut} in the appendix for more analyses).


\begin{table}[!h]
\centering
\resizebox{0.95\textwidth}{!}{%
\begin{tabular}{ccc|ccccc|rrrrr|}
\cline{4-13}
\multicolumn{1}{l}{} &
  \multicolumn{1}{l}{} &
  \multicolumn{1}{l|}{} &
  \multicolumn{5}{c|}{\#(Solved instances)} &
  \multicolumn{5}{c|}{Computing time (s)} \\ \hline
\multicolumn{1}{|c|}{} &
  \multicolumn{1}{c|}{} &
   &
  \multicolumn{1}{l|}{} &
  \multicolumn{2}{c|}{B\&C} &
  \multicolumn{2}{c|}{CP} &
  \multicolumn{1}{l|}{} &
  \multicolumn{2}{c|}{B\&C} &
  \multicolumn{2}{c|}{CP} \\ \cline{5-8} \cline{10-13} 
\multicolumn{1}{|c|}{Dataset} &
  \multicolumn{1}{l|}{$v_{i0}$} &
  $\alpha$ &
  \multicolumn{1}{l|}{CONIC} &
  \multicolumn{1}{l}{\eqref{Bi-master-SB}} &
  \multicolumn{1}{l|}{\eqref{Li-master-SB}} &
  \multicolumn{1}{l}{\eqref{Bi-master-SB}} &
  \multicolumn{1}{l|}{\eqref{Li-master-SB}} &
  \multicolumn{1}{r|}{CONIC} &
  \eqref{Bi-master-SB} &
  \multicolumn{1}{r|}{\eqref{Li-master}} &
  \eqref{Bi-master-SB} &
  \eqref{Li-master-SB} \\ \hline
\multicolumn{1}{|c|}{\multirow{6}{*}{$\texttt{100\_1000}$}} &
  \multicolumn{1}{c|}{\multirow{3}{*}{1}} &
  10 &
  \multicolumn{1}{c|}{\textbf{5}} &
  \textbf{5} &
  \multicolumn{1}{c|}{\textbf{5}} &
  \textbf{5} &
  \textbf{5} &
  \multicolumn{1}{r|}{301.06} &
  169.68 &
  \multicolumn{1}{r|}{408.98} &
  \textbf{6.71} &
  6.93 \\
\multicolumn{1}{|c|}{} &
  \multicolumn{1}{c|}{} &
  20 &
  \multicolumn{1}{c|}{\textbf{5}} &
  \textbf{5} &
  \multicolumn{1}{c|}{\textbf{5}} &
  \textbf{5} &
  \textbf{5} &
  \multicolumn{1}{r|}{146.43} &
  76.20 &
  \multicolumn{1}{r|}{81.96} &
  \textbf{4.19} &
  13.88 \\
\multicolumn{1}{|c|}{} &
  \multicolumn{1}{c|}{} &
  50 &
  \multicolumn{1}{c|}{\textbf{5}} &
  \textbf{5} &
  \multicolumn{1}{c|}{\textbf{5}} &
  \textbf{5} &
  \textbf{5} &
  \multicolumn{1}{r|}{104.55} &
  \textbf{27.64} &
  \multicolumn{1}{r|}{48.28} &
  414.52 &
  667.45 \\ \cline{2-13} 
\multicolumn{1}{|c|}{} &
  \multicolumn{1}{c|}{\multirow{3}{*}{2}} &
  10 &
  \multicolumn{1}{c|}{\textbf{5}} &
  \textbf{5} &
  \multicolumn{1}{c|}{\textbf{5}} &
  \textbf{5} &
  \textbf{5} &
  \multicolumn{1}{r|}{511.27} &
  281.53 &
  \multicolumn{1}{r|}{653.68} &
  \textbf{14.17} &
  14.50 \\
\multicolumn{1}{|c|}{} &
  \multicolumn{1}{c|}{} &
  20 &
  \multicolumn{1}{c|}{\textbf{5}} &
  \textbf{5} &
  \multicolumn{1}{c|}{\textbf{5}} &
  \textbf{5} &
  \textbf{5} &
  \multicolumn{1}{r|}{289.16} &
  111.83 &
  \multicolumn{1}{r|}{205.67} &
  \textbf{4.18} &
  5.05 \\
\multicolumn{1}{|c|}{} &
  \multicolumn{1}{c|}{} &
  50 &
  \multicolumn{1}{c|}{\textbf{5}} &
  \textbf{5} &
  \multicolumn{1}{c|}{\textbf{5}} &
  \textbf{5} &
  \textbf{5} &
  \multicolumn{1}{r|}{104.00} &
  \textbf{28.31} &
  \multicolumn{1}{r|}{35.26} &
  519.02 &
  545.61 \\ \hline
\multicolumn{1}{|c|}{\multirow{6}{*}{$\texttt{200\_2000}$}} &
  \multicolumn{1}{c|}{\multirow{3}{*}{5}} &
  10 &
  \multicolumn{1}{c|}{0} &
  4 &
  \multicolumn{1}{c|}{2} &
  \textbf{5} &
  \textbf{5} &
  \multicolumn{1}{r|}{-} &
  1884.29 &
  \multicolumn{1}{r|}{1243.95} &
  \textbf{60.82} &
  65.63 \\
\multicolumn{1}{|c|}{} &
  \multicolumn{1}{c|}{} &
  20 &
  \multicolumn{1}{c|}{0} &
  \textbf{5} &
  \multicolumn{1}{c|}{\textbf{5}} &
  \textbf{5} &
  \textbf{5} &
  \multicolumn{1}{r|}{-} &
  1060.01 &
  \multicolumn{1}{r|}{1839.42} &
  \textbf{20.03} &
  22.30 \\
\multicolumn{1}{|c|}{} &
  \multicolumn{1}{c|}{} &
  50 &
  \multicolumn{1}{c|}{2} &
  \textbf{5} &
  \multicolumn{1}{c|}{\textbf{5}} &
  \textbf{5} &
  \textbf{5} &
  \multicolumn{1}{r|}{3075.02} &
  499.13 &
  \multicolumn{1}{r|}{537.40} &
  26.34 &
  \textbf{24.81} \\ \cline{2-13} 
\multicolumn{1}{|c|}{} &
  \multicolumn{1}{c|}{\multirow{3}{*}{10}} &
  10 &
  \multicolumn{1}{c|}{0} &
  \textbf{5} &
  \multicolumn{1}{c|}{\textbf{5}} &
  \textbf{5} &
  \textbf{5} &
  \multicolumn{1}{r|}{-} &
  1863.66 &
  \multicolumn{1}{r|}{1646.67} &
  \textbf{45.77} &
  54.10 \\
\multicolumn{1}{|c|}{} &
  \multicolumn{1}{c|}{} &
  20 &
  \multicolumn{1}{c|}{0} &
  \textbf{5} &
  \multicolumn{1}{c|}{\textbf{5}} &
  \textbf{5} &
  \textbf{5} &
  \multicolumn{1}{r|}{-} &
  881.53 &
  \multicolumn{1}{r|}{1100.94} &
  \textbf{19.13} &
  19.39 \\
\multicolumn{1}{|c|}{} &
  \multicolumn{1}{c|}{} &
  50 &
  \multicolumn{1}{c|}{0} &
  \textbf{5} &
  \multicolumn{1}{c|}{\textbf{5}} &
  \textbf{5} &
  \textbf{5} &
  \multicolumn{1}{r|}{-} &
  408.05 &
  \multicolumn{1}{r|}{525.99} &
  23.89 &
  \textbf{22.18} \\ \hline
\multicolumn{1}{|c|}{\multirow{6}{*}{$\texttt{200\_4000}$}} &
  \multicolumn{1}{c|}{\multirow{3}{*}{5}} &
  10 &
  \multicolumn{1}{c|}{0} &
  1 &
  \multicolumn{1}{c|}{1} &
  \textbf{5} &
  \textbf{5} &
  \multicolumn{1}{r|}{-} &
  2888.00 &
  \multicolumn{1}{r|}{2393.56} &
  \textbf{117.01} &
  132.87 \\
\multicolumn{1}{|c|}{} &
  \multicolumn{1}{c|}{} &
  20 &
  \multicolumn{1}{c|}{0} &
  2 &
  \multicolumn{1}{c|}{2} &
  \textbf{5} &
  \textbf{5} &
  \multicolumn{1}{r|}{-} &
  2203.68 &
  \multicolumn{1}{r|}{1501.80} &
  \textbf{93.47} &
  127.05 \\
\multicolumn{1}{|c|}{} &
  \multicolumn{1}{c|}{} &
  50 &
  \multicolumn{1}{c|}{0} &
  \textbf{5} &
  \multicolumn{1}{c|}{\textbf{5}} &
  \textbf{5} &
  \textbf{5} &
  \multicolumn{1}{r|}{-} &
  1435.87 &
  \multicolumn{1}{r|}{1045.85} &
  \textbf{33.16} &
  46.62 \\ \cline{2-13} 
\multicolumn{1}{|c|}{} &
  \multicolumn{1}{c|}{\multirow{3}{*}{10}} &
  10 &
  \multicolumn{1}{c|}{0} &
  3 &
  \multicolumn{1}{c|}{2} &
  \textbf{5} &
  \textbf{5} &
  \multicolumn{1}{r|}{-} &
  2338.95 &
  \multicolumn{1}{r|}{2191.28} &
  \textbf{66.27} &
  71.75 \\
\multicolumn{1}{|c|}{} &
  \multicolumn{1}{c|}{} &
  20 &
  \multicolumn{1}{c|}{0} &
  3 &
  \multicolumn{1}{c|}{3} &
  \textbf{5} &
  \textbf{5} &
  \multicolumn{1}{r|}{-} &
  2535.00 &
  \multicolumn{1}{r|}{2739.70} &
  \textbf{100.30} &
  121.97 \\
\multicolumn{1}{|c|}{} &
  \multicolumn{1}{c|}{} &
  50 &
  \multicolumn{1}{c|}{0} &
  \textbf{5} &
  \multicolumn{1}{c|}{3} &
  \textbf{5} &
  \textbf{5} &
  \multicolumn{1}{r|}{-} &
  1574.44 &
  \multicolumn{1}{r|}{1040.74} &
  \textbf{34.01} &
  36.98 \\ \hline
\multicolumn{1}{|c|}{\multirow{6}{*}{$\texttt{100\_5000}$}} &
  \multicolumn{1}{c|}{\multirow{3}{*}{1}} &
  10 &
  \multicolumn{1}{c|}{2} &
  \textbf{5} &
  \multicolumn{1}{c|}{2} &
  \textbf{5} &
  \textbf{5} &
  \multicolumn{1}{r|}{2465.84} &
  1688.00 &
  \multicolumn{1}{r|}{2898.66} &
  \textbf{6.42} &
  6.47 \\
\multicolumn{1}{|c|}{} &
  \multicolumn{1}{c|}{} &
  20 &
  \multicolumn{1}{c|}{\textbf{5}} &
  \textbf{5} &
  \multicolumn{1}{c|}{\textbf{5}} &
  \textbf{5} &
  \textbf{5} &
  \multicolumn{1}{r|}{1603.59} &
  1288.66 &
  \multicolumn{1}{r|}{936.60} &
  \textbf{24.22} &
  30.42 \\
\multicolumn{1}{|c|}{} &
  \multicolumn{1}{c|}{} &
  50 &
  \multicolumn{1}{c|}{3} &
  \textbf{5} &
  \multicolumn{1}{c|}{\textbf{5}} &
  \textbf{5} &
  3 &
  \multicolumn{1}{r|}{1798.74} &
  \textbf{364.26} &
  \multicolumn{1}{r|}{1145.03} &
  1290.35 &
  1606.71 \\ \cline{2-13} 
\multicolumn{1}{|c|}{} &
  \multicolumn{1}{c|}{\multirow{3}{*}{2}} &
  10 &
  \multicolumn{1}{c|}{0} &
  \textbf{5} &
  \multicolumn{1}{c|}{2} &
  \textbf{5} &
  \textbf{5} &
  \multicolumn{1}{r|}{-} &
  2295.70 &
  \multicolumn{1}{r|}{2995.34} &
  8.10 &
  \textbf{7.74} \\
\multicolumn{1}{|c|}{} &
  \multicolumn{1}{c|}{} &
  20 &
  \multicolumn{1}{c|}{2} &
  \textbf{5} &
  \multicolumn{1}{c|}{\textbf{5}} &
  \textbf{5} &
  \textbf{5} &
  \multicolumn{1}{r|}{3090.49} &
  2020.56 &
  \multicolumn{1}{r|}{1248.85} &
  \textbf{18.68} &
  25.30 \\
\multicolumn{1}{|c|}{} &
  \multicolumn{1}{c|}{} &
  50 &
  \multicolumn{1}{c|}{\textbf{5}} &
  \textbf{5} &
  \multicolumn{1}{c|}{\textbf{5}} &
  \textbf{5} &
  3 &
  \multicolumn{1}{r|}{1532.11} &
  \textbf{312.29} &
  \multicolumn{1}{r|}{449.97} &
  1099.63 &
  621.86 \\ \hline
  \multicolumn{3}{|r|}{Summary} &  \multicolumn{1}{c|}{59} & 108 & \multicolumn{1}{c|}{99} & \textbf{120} & 116 & \multicolumn{1}{r|}{721.10}  & 1031.23  & \multicolumn{1}{r|}{971.65} & 168.77  &  \textbf{146.82}\\ \hline
\end{tabular}%
}
\caption{Numerical results for instances of large numbers of customer classes.}
\label{tab:large_customer}
\end{table}

Table \ref{tab:large_customer} presents our comparison results for both the B\&C and CP approaches, based on the two segment-based master programs \eqref{Bi-master-SB} and \eqref{Li-master-SB}. We utilize both OA and SC cuts for our methods and include the CONIC methods for comparison. The best performance is achieved by \eqref{Bi-master-SB} embedded in the CP algorithm, which returns optimal solutions for all instances, followed by the CP with the linear master program \eqref{Li-master-SB} (116/120). The B\&C methods with \eqref{Bi-master-SB} and \eqref{Li-master-SB} provide 108/120 and 99/120 optimal solutions, respectively, while the CONIC method solves only 59 out of 120 instances. In terms of computing time, the CP with \eqref{Bi-master-SB} is the fastest method for 17 out of 24 sets of instances, being approximately 4.27 times faster than the CONIC and 6.11 times faster than the B\&C with \eqref{Bi-master-SB}, which provides the highest number of optimal solutions among the B\&C methods. The experiment demonstrates that the CP algorithm outperforms other approaches (CONIC and B\&C) in handling instances with a large number of customers.

\begin{figure}[!h]
    \centering
    \includegraphics[width=\linewidth]{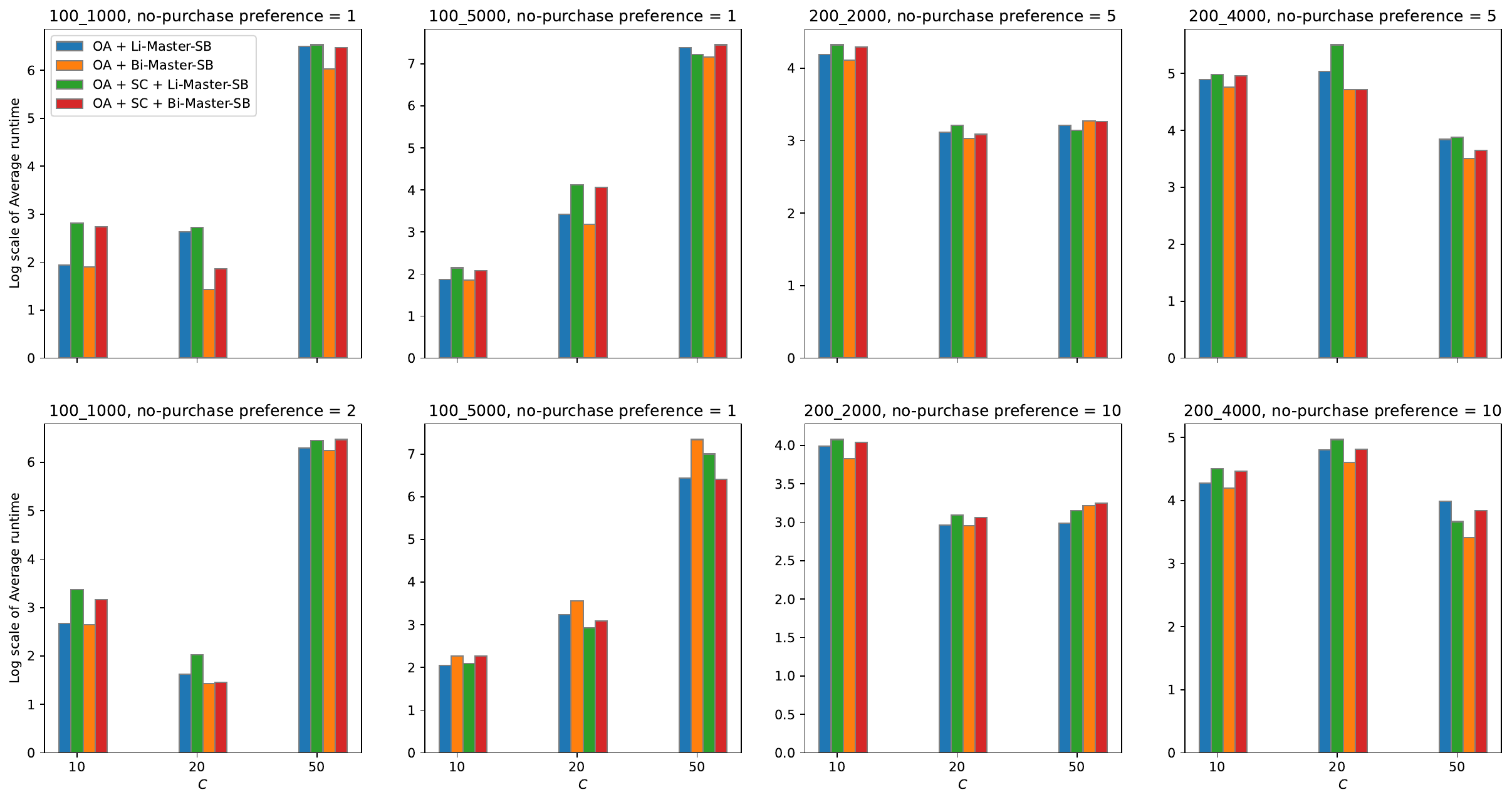}
    \caption{Comparison of computation times for segment-based CP algorithms with OA and SC cuts.}
    \label{fig:oa_oasc}
\end{figure}

Given that segment-based CP approaches offer the best performance when solving instances of large customer classes, we extend our analysis by examining the impact of OA and SC cuts on the overall performance of the CP algorithm. Figure \ref{fig:oa_oasc} visualizes the average runtime of the CP algorithm for four different settings: CP under \eqref{Bi-master-SB} with OA cuts, CP under \eqref{Bi-master-SB} with OA+SC cuts, CP under \eqref{Li-master-SB} with OA cuts, and CP under \eqref{Li-master-SB} with OA+SC cuts. Each bar represents the average runtime's logarithmic scale for each setting over 120 large instances.

Overall, the OA+SC versions (green and red bars) run slightly slower than the models with only OA cuts (blue and orange bars) in almost all cases, except for some instances with capacity equal to 50 solved using \eqref{Li-master-SB}. This phenomenon also occurs in the case of the B\&C with \eqref{Bi-master-SB}/\eqref{Li-master-SB} as shown in Tables \ref{tab:200_20_k} - \ref{tab:1000_100_k}.


\section{Conclusion}\label{sec:conclusion}
In this paper, we have addressed the capacitated assortment problem under the MMNL discrete choice model. Unlike previous methods that primarily rely on MILP or CONIC reformulations, we have explored nonlinear components of objective functions that are convex and super-modular. We developed two types of cuts—outer-approximation and super-modular cuts—to create new CP and B\&C approaches, which iteratively add such cuts to a master problem (linear or bilinear). Additionally, we theoretically demonstrate that super-modularity can be leveraged to devise a simple polynomial-time Greedy Heuristic algorithm that yields solutions with semi-constant factor guarantees. Experiments conducted on instances of various sizes establish the superiority of our CP and B\&C approaches, which can solve instances with a very large number of products and customer classes to optimality, outperforming the state-of-the-art CONIC approach.

\bibliographystyle{apacite}
\bibliography{refs}

\clearpage

\appendix
\section*{APPENDIX}

\section{Additional Numerical Results}
\subsection{Comparison of B\&C and CP Approaches}
As mentioned in the main paper, the B\&C consistently outperforms the CP approaches for small- and medium-sized instances when the number of customer classes is not too large. Below, we provide some numerical comparisons to illustrate this. Table \ref{tab:new_formulations} presents a comparison between the B\&C with \eqref{Li-master}—the best approach as shown in Tables \ref{tab:capacity_only} - \ref{tab:1000_100_k}—and the CP methods under \eqref{Bi-master} and \eqref{Li-master}, using some standard benchmark instances from \cite{Sen2018}. It can be observed that the CP performs poorly on instances with large capacity, while the B\&C with \eqref{Li-master} solves all instances to optimality.

\begin{table}[htb]
\centering
\resizebox{0.95\textwidth}{!}{%
\begin{tabular}{|c|c|c|c|cc|c|cc|} 
\cline{4-9}
\multicolumn{1}{l}{}                         & \multicolumn{1}{l}{}          & \multicolumn{1}{l|}{} & \multicolumn{3}{c|}{\#(Solved instances)}                    & \multicolumn{3}{c|}{Computing time (s)}                             \\ 
\hline
                                             &                               &                       & B\&C + \eqref{Li-master}        & \multicolumn{2}{c|}{CP}                       & BC+ \eqref{Li-master}          & \multicolumn{2}{c|}{CP}                            \\ 
\cline{5-6}\cline{8-9}
Dataset                                      & \multicolumn{1}{l|}{$v_{i0}$} & $\alpha(\alpha_k)$    &              & \multicolumn{1}{l}{\eqref{Bi-master}} & \multicolumn{1}{l|}{\eqref{Li-master}}  &                & \multicolumn{1}{l}{\eqref{Bi-master}} & \multicolumn{1}{l|}{\eqref{Li-master}}       \\ 
\hline
\multirow{10}{*}{$\texttt{Sen\_200\_20}$}    & \multirow{5}{*}{5}            & 10                    & \textbf{5}   & \textbf{5}           & \textbf{5}             & 4.44           & \textbf{1.06}        & 1.21                        \\
                                             &                               & 20                    & \textbf{5}   & \textbf{5}           & \textbf{5}             & 3.12           & \textbf{1.51}        & 2.08                        \\
                                             &                               & 50                    & \textbf{5}   & \textbf{5}           & \textbf{5}             & \textbf{1.13}  & 10.91                & 21.77                       \\
                                             &                               & 100                   & \textbf{5}   & 0                    & 0                      & \textbf{0.53}  & -                    & -                           \\
                                             &                               & 200                   & \textbf{5}   & 0                    & 0                      & \textbf{0.56}  & -                    & -                           \\ 
\cline{2-9}
                                             & \multirow{5}{*}{10}           & 10                    & \textbf{5}   & \textbf{5}           & \textbf{5}             & 4.11           & \textbf{0.65}        & 0.66                        \\
                                             &                               & 20                    & \textbf{5}   & \textbf{5}           & \textbf{5}             & 3.25           & \textbf{0.90}        & 1.16                        \\
                                             &                               & 50                    & \textbf{5}   & \textbf{5}           & \textbf{5}             & \textbf{1.96}  & 9.81                 & 14.55                       \\
                                             &                               & 100                   & \textbf{5}   & 0                    & 0                      & \textbf{0.54}  & -                    & -                           \\
                                             &                               & 200                   & \textbf{5}   & 0                    & 0                      & \textbf{0.70}  & -                    & -                           \\ 
\hline
\multirow{8}{*}{$\texttt{Sen\_100\_100}$}    & \multirow{4}{*}{1}            & 10                    & \textbf{5}   & \textbf{5}           & \textbf{5}             & \textbf{30.36} & 124.18               & 181.51                      \\
                                             &                               & 20                    & \textbf{5}   & \textbf{4}           & \textbf{4}             & \textbf{50.95} & 928.38               & 1042.14                     \\
                                             &                               & 50                    & \textbf{5}   & 0                    & 0                      & \textbf{4.31}  & -                    & -                           \\
                                             &                               & 100                   & \textbf{5}   & 0                    & 0                      & \textbf{2.22}  & -                    & -                           \\ 
\cline{2-9}
                                             & \multirow{4}{*}{2}            & 10                    & \textbf{5}   & \textbf{5}           & \textbf{5}             & \textbf{16.55} & 22.32                & 25.26                       \\
                                             &                               & 20                    & \textbf{5}   & \textbf{5}           & \textbf{5}             & \textbf{27.95} & 181.49               & 282.79                      \\
                                             &                               & 50                    & \textbf{5}   & \textbf{2}           & 1                      & \textbf{8.24}  & 2285.74              & 2858.20                     \\
                                             &                               & 100                   & \textbf{5}   & \textbf{5}           & 4                      & \textbf{1.03}  & 678.32               & 373.84                      \\ 
\hline
\multirow{10}{*}{$\texttt{Sen\_200\_20\_5}$} & \multirow{5}{*}{10}           & 5(2)                  & \textbf{5}   & \textbf{5}           & \textbf{5}             & 3.47           & \textbf{0.66}        & 0.67                        \\
                                             &                               & 10(4)                 & \textbf{5}   & \textbf{5}           & \textbf{5}             & \textbf{4.55}  & \textbf{2.18}        & 3.54                        \\
                                             &                               & 25(10)                & \textbf{5}   & \textbf{5}           & \textbf{5}             & \textbf{6.43}  & 33.41                & 289.79                      \\
                                             &                               & 50(20)                & \textbf{5}   & 0                    & 0                      & \textbf{1.02}  & -                    & -                           \\
                                             &                               & 100(40)               & \textbf{5}   & 0                    & 0                      & \textbf{0.80}  & -                    & -                           \\ 
\cline{2-9}
                                             & \multirow{5}{*}{20}           & 5(2)                  & \textbf{5}   & \textbf{5}           & \textbf{5}             & 2.49           & 0.55                 & 0.53                        \\
                                             &                               & 10(4)                 & \textbf{5}   & \textbf{5}           & \textbf{5}             & 3.09           & \textbf{1.84}        & 2.66                        \\
                                             &                               & 25(10)                & \textbf{5}   & \textbf{5}           & \textbf{5}             & \textbf{6.26}  & 13.41                & 39.34                       \\
                                             &                               & 50(20)                & \textbf{5}   & 0                    & 0                      & \textbf{2.06}  & -                    & -                           \\
                                             &                               & 100(40)               & \textbf{5}   & 0                    & 0                      & \textbf{0.70}  & -                    & -                           \\ 
\hline
\multicolumn{3}{|r|}{Summary}                                                                        & \textbf{140} & 86                   & \multicolumn{1}{c}{84} & \textbf{6.89}  & 159.31               & \multicolumn{1}{r|}{153.09}  \\
\hline
\end{tabular}}
\caption{Comparison of B\&C and CP for small- and medium-sized instances. }
\label{tab:new_formulations}
\end{table}

\subsection{Impact of the Number of Cuts for the Segment-based Approach}

In this section, we explore the impact of the number of customer groups ($L$) on the performance of the CP and B\&C methods under the segment-based master problem \eqref{Bi-master-SB}, using instances with a large number of customer classes described in Section 5.3. Figure \ref{fig:bcvscp} reports the average computation times of the CP and B\&C methods as $L$ varies in $\{1, 5, 10, 20, 50, 100, 200\}$. The figure shows that as the number of cuts increases from 20 to 200, the runtime experiences a slight upward trend, with the CP method achieving the best performance at $L = 20$. In contrast, the computation times of the B\&C method are unstable when $L$ increases from 1 to 50, and then rapidly increase again as $L$ goes from 50 to 200. Furthermore, Figure \ref{fig:bcvscp} indicates that the CP algorithms (red and green lines) significantly outperform and run more stably than the B\&C methods (blue and orange lines) as $L$ increases.


\begin{figure}[htb]
    \centering
    \includegraphics[width=0.95\linewidth]{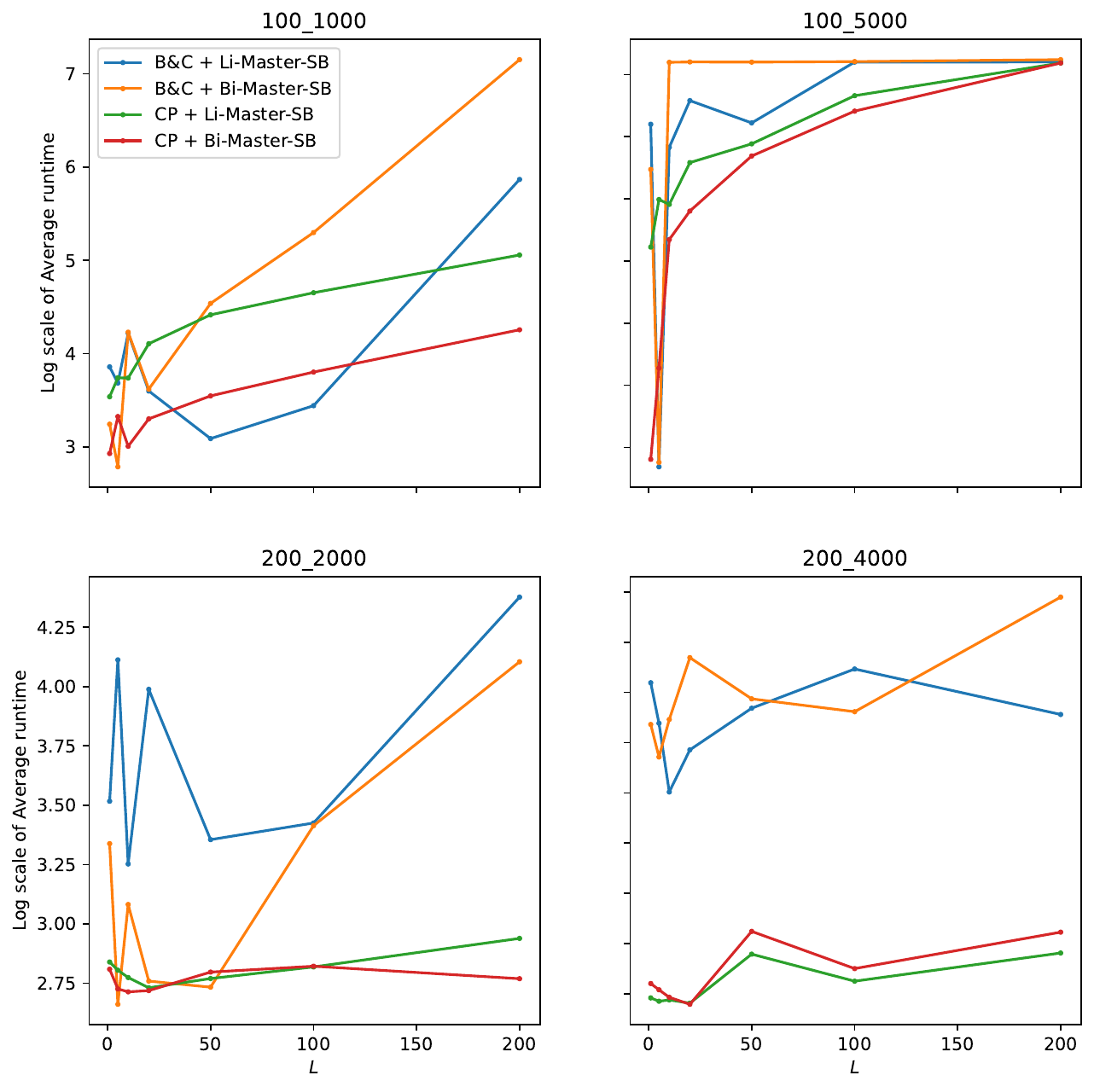}
    \caption{Computation time comparison with varying number of customer groups $L$.}
    \label{fig:bcvscp}
\end{figure}

\end{document}